\documentclass{article}
\usepackage[utf8]{inputenc}
\usepackage{amsmath}
\usepackage{amssymb}
\usepackage{tikz}
\usepackage{hyperref}
\usepackage{graphicx}
\usepackage{algorithmic}
\usepackage{algorithm}
\usepackage{multirow}
\usepackage{subfig}
\usepackage[noabbrev, capitalise]{cleveref}
\usepackage{authblk}
\usepackage[toc,page]{appendix}
\usepackage{amsthm,thmtools}
\usepackage{thm-restate}
\usepackage{float}
\usetikzlibrary{positioning}

\theoremstyle{plain}
\declaretheorem[name=Proposition, numberwithin=section]{prop}
\declaretheorem[name=Corollary, numberwithin=prop]{corollary}

\title{Gaussian belief propagation solvers for nonsymmetric systems of linear equations}
\author{Vladimir Fanaskov\thanks{Vladimir.Fanaskov@skoltech.ru}}
\affil{Center for Design, Manufacturing, and Materials, Skolkovo Institute of Science and Technology,
Bolshoy Boulevard 30, bld. 1,
Moscow, Russia, 121205.}
\date{}
\begin{document}
\maketitle
\begin{abstract}
In this paper, we argue for the utility of deterministic inference in the classical problem of numerical linear algebra, that of solving a linear system. We show how the Gaussian belief propagation solver, known to work for symmetric matrices can be modified to handle nonsymmetric matrices. Furthermore, we introduce a new algorithm for matrix inversion that corresponds to the generalized belief propagation derived from the cluster variation method (or Kikuchi approximation). We relate these algorithms to LU and block LU decompositions and provide certain guarantees based on theorems from the theory of belief propagation. All proposed algorithms are compared with classical solvers (e.g., Gauss-Seidel, BiCGSTAB) with application to linear elliptic equations. We also show how the Gaussian belief propagation can be used as multigrid smoother, resulting in a substantially more robust solver than the one based on the Gauss-Seidel iterative method.
\end{abstract}
\section{Introduction}
A basic problem of numerical linear algebra is to solve a linear equation 
\begin{equation}\label{linear_problem}
    \mathbf{A}\mathbf{x} = \mathbf{b}
\end{equation}
with an invertible matrix $\mathbf{A}$. The textbook technique is the LU decomposition, equivalent to Gaussian elimination \cite[ch. 3]{golub2012matrix}. However, when $\mathbf{A}$ is large and sparse, algorithms that exploit sparsity are used instead of the direct elimination \cite{davis2016survey}. Among iterative methods for sparse systems, one can mention classical relaxation techniques such as  Gauss-Seidel (GS), Jacobi, Richardson, and projection methods such as conjugate gradients, generalized minimal residuals, biconjugate gradients, and others \cite{saad2001iterative,saad2003iterative}. 

An easy way to understand the projection methods is to reformulate the original equation as an optimization problem \cite{shewchuk1994introduction}. For example, for a symmetric positive-definite matrix $\mathbf{A}$, one has  
\begin{equation}\label{optimization_problem}
    \mathbf{x}^{\star} = \arg\min_{\mathbf{x}}\left(\frac{\mathbf{x}^T \mathbf{A} \mathbf{x}}{2} - \mathbf{x}^T \mathbf{b}\right).
\end{equation}
Such a reformulation allows one to apply new techniques and leads to methods of steepest descent, conjugate directions, and cheap and efficient conjugate gradients \cite{Hestenes&Stiefel:1952}. 

Another reformulation of the problem is known, but is less explored. It also goes back to Gauss and his version of elimination. To derive an LU solution of \eqref{linear_problem}, one can consider the probability density function $p(\mathbf{x})$ of multivariate normal distribution (see also equation~(\ref{Gauss_Markov_model})  below)
\begin{equation}\label{multivariate_normal}
    p(\mathbf{x})\sim\exp\left(-\frac{\mathbf{x}^T \mathbf{A} \mathbf{x}}{2} + \mathbf{b}^T\mathbf{x}\right).
\end{equation}
We can consider the first component $x_1$ of $\mathbf{x}$
%, which is $x_1$, 
and integrate it out in $(\ref{multivariate_normal})$ (a process called ``marginalization''). The resulting marginal distribution for the remaining components $x_2,x_3,\dots$ is again multivariate normal, but with the covariance matrix given by the Schur complement of $A_{11}$\footnote{In the article we use boldface for matrices or matrix blocks, and regular font for scalar values and matrix components. In this case $A_{11}$ is an element of the matrix $\mathbf{A}$ in the first row and the first column, and $\mathbf{A_{22}}$ is a square matrix that contains all elements of $\mathbf{A}$ excluding the first row and the first column.} and the mean vector modified accordingly, i.e. 
\begin{equation}
    \mathbf{A_{22}} \leftarrow \mathbf{A_{22}} - \frac{\mathbf{A}_{\mathbf{2}1}\mathbf{A}_{1\mathbf{2}}}{A_{11}},~\mathbf{b_2} \leftarrow \mathbf{b_2} - \frac{\mathbf{A}_{\mathbf{2}1} b_1}{A_{11}}.
\end{equation}
It is well known that the LU decomposition consists of the very same steps \cite{stewart1998matrix}. When $x_1$ is not a scalar, but a subset of variables, marginalization of multivariate normal distribution results in a block LU decomposition.

Thus, the most popular direct technique for the solution of linear equations with dense matrices is intimately connected with the marginalization problem, which belongs to the class of inference problems. Recently, many other intriguing connections between statistical inference and linear algebra have been pointed out. For instance, in  \cite{cockayne2018bayesian}, \cite{hennig2015probabilistic}, and  \cite{bartels2018probabilistic}, the authors provide a method to recover the Petrov-Galerkin condition from the Bayesian update and construct a Bayesian version of the conjugate gradients. Paper \cite{owhadi2017multigrid} constructs a state-of-the-art multigrid solver using the game theory and statistical inference. These works demonstrate that ideas from statistical inference allow for new and useful insights into problems of linear algebra. It is thus reasonable to explore how other inference algorithms are translated to the realm of numerical linear algebra. Among them are expectation propagation \cite{minka2001expectation}, Markov chain Monte Carlo, mean field,  other variational Bayesian approximations \cite[chapters 8, 10]{Bishop:2006:PRM:1162264}, \cite{wainwright2008graphical}, and belief propagation with its generalized counterparts. The latter two are the focus of the present work.

The first comparison between classical methods and belief propagation appeared in  \cite{weiss2000correctness}. Then in \cite{shental2008gaussian}, authors argued explicitly for the belief propagation as a solver and later, Bickson \cite{bickson2008gaussian} presented a more systematic treatment of the Gaussian belief propagation (GaBP) in the same context. Among other proposed methods was an algorithm that treats nonsymmetric matrices through diagonal weighting \cite[5.4]{bickson2008gaussian} and the usual trick from linear algebra, $\mathbf{A}\rightarrow \mathbf{A}^T\mathbf{A}$. Both techniques are of limited use because of slow convergence in the first case and fill-in in the second. We improve on these results and propose several new algorithms.

In particular, in this work we:
\begin{itemize}
    \item explain how belief propagation can be applied to nonsymmetric matrices with no computation overhead compared to the original belief propagation (which was limited to symmetric matrices) (\cref{algorithm:non_symmetric_GaBP});
    \item design a family of linear solvers based on the generalized belief propagation (\cref{algorithm:non_symmetric_two_layers_GaBP}) and relate them to the block LU decomposition (see \cref{Section:Generalized_GaBP_solvers.Subsection:Elimination_perspective});
    \item introduce a two-layer region graph and derive generalized belief propagation rules (\ref{Two_layers_GaBP_rules}) that are substantially less  demanding computationally compared to the basic  generalized belief propagation algorithm (see (\ref{computational_gain}));
    \item show how proofs of sufficient condition for convergence and consistency for the original GaBP can be modified to hold for the new algorithms (for GaBP see \ref{Appendices:Consistency_of_GaBP}, \ref{Appendices:Convergence_of_GaBP}, for generalized GaBP -- \ref{Appendices:Consistency_of_generalized_GaBP}, \ref{Appendices:Convergence_of_generalized_GaBP});
    \item explain how one can speed up GaBP using multigrid methodology which results in a robust solver (see \cref{fig:convergence_anisotropic_problem});
    \item implement the new algorithms \cite{git_GaBP} and benchmark them against several classical multigrid solvers.
\end{itemize}
%rewrite
The rest of the paper is organized as follows. In \cref{Section:Gaussian_belief_propagation.Subsection:GaBP_from_the_elimination_perspective}, we start with the intuitive explanation of GaBP based on the connection between the algorithm and Gaussian elimination. The general description of how to treat the problem (\ref{linear_problem}) as an inference problem together with the basic terminology and main facts about Gaussian belief propagation are introduced in \cref{Section:Gaussian_belief_propagation}.
In \cref{Section:Gaussian_belief_propagation.Section:GaBP_for_the_nonsymmetric_linear_system} we introduce GaBP that can be used for nonsymmetric matrices, prove consistency of the proposed algorithm, and establish a sufficient condition for convergence in appendices  \ref{Appendices:Consistency_of_GaBP}, \ref{Appendices:Convergence_of_GaBP}. In \cref{Section:Generalized_GaBP_solvers}, extensions of the belief propagation are given: in \ref{Section:Generalized_GaBP_solvers.Subsection:Parent-to-child_algorithm}, we describe the generalized belief propagation algorithm (GaBP) (parent-to-child in \cite{yedidia2005constructing}); in \ref{Section:Generalized_GaBP_solvers.Subsection:Generalized_two_layers_GaBP}, we derive message update rules for region graphs with two layers, and in \ref{Section:Generalized_GaBP_solvers.Subsection:Elimination_perspective}, we discuss the generalized GaBP from the elimination perspective and explain why the algorithm can be applied to the case $\mathbf{A}^{T}\neq \mathbf{A}$; the resulting algorithm is introduced in \cref{Section:Generalized_GaBP_solvers.Subsection:The_algorithm} and analyzed in \ref{Appendices:Consistency_of_generalized_GaBP}, \ref{Appendices:Convergence_of_generalized_GaBP}. \Cref{Section:GaBP_as_a_smoother_for_the_multigrid_scheme} explains how to use GaBP within a multigrid scheme, we discuss smoothing properties, complexity and describe a rather  unusual behaviour for singularly perturbed elliptic equations. \Cref{Section:Numerical_examples} contains numerical examples. In \cref{Section:Conclusion}, we summarize the main results and discuss possible future research.

\section{Gaussian belief propagation}\label{Section:Gaussian_belief_propagation}
\subsection{GaBP from the elimination perspective}\label{Section:Gaussian_belief_propagation.Subsection:GaBP_from_the_elimination_perspective}
\begin{figure}
    \centering
    \subfloat[]{
    \label{fig:elimination:a}
        \begin{tikzpicture}[
            sq/.style={circle, draw=black, minimum size=1mm},
            ]
            \node[sq] at (0, 0) (5) {$5$};
            \node[] at (-1, -0.5) (5_3) {$M_{53}$};
            \node[sq] at (-1, -1.5) (3) {$3$};
            \node[sq] at (1, -1.5) (4) {$4$};
            \node[] at (1, -0.5) (5_4) {$M_{54}$};
            \node[sq] at (-2, -3) (1) {$1$};
            \node[] at (-2, -2) (3_1) {$M_{31}$};
            \node[sq] at (0, -3) (2) {$2$};
            \node[] at (0, -2) (3_2) {$M_{32}$};
            \draw[->] (3) -- (5);
            \draw[->] (1) -- (3);
            \draw[->] (2) -- (3);
            \draw[->] (4) -- (5);
        \end{tikzpicture}
    }\quad
    \subfloat[]{
    \label{fig:elimination:b}
        \begin{tikzpicture}[
            sq/.style={circle, draw=black, minimum size=1mm},
            ]
            \node[sq] at (0, 0) (3) {$3$};
            \node[sq] at (-1.5, -1.5) (1) {$1$};
            \node[] at (-1.2, -0.5) (3_1) {$M_{31}$};
            \node[sq] at (0, -2) (2) {$2$};
            \node[] at (0.5, -1.3) (3_2) {$M_{32}$};
            \node[sq] at (1.5, -1.5) (5) {$5$};
            \node[] at (1.2, -0.5) (3_5) {$M_{35}$};
            \node[sq] at (1.5, -3) (4) {$4$};
            \node[] at (2, -2.3) (5_4) {$M_{54}$};
            \draw[->] (5) -- (3);
            \draw[->] (4) -- (5);
            \draw[->] (1) -- (3);
            \draw[->] (2) -- (3);
        \end{tikzpicture}
    }
    \caption{Both (a) and (b) sketch the graph, corresponding to the matrix $\mathbf{A}$ fron equation~(\ref{toy_GaBP}). We use $M_{ji}$ to represent the pair of messages $\left(\Lambda_{ji}, \mu_{ji}\right)$ (see equation~(\ref{toy_messages})) from the node $i$ to $j$. Figures (a)  and (b) shows different order of elimination. For example, in case of (a) one first exclude $x_1$ and $x_2$ from the equation for $x_3$ and then solve resulting equation to obtain $x_5$.}
    \label{fig:elimination}
\end{figure}
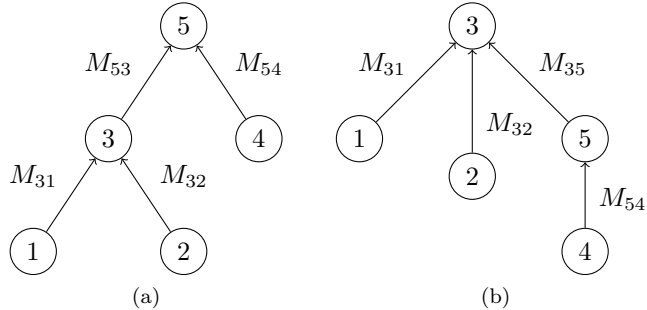
To build intuition about GaBP, we consider connection with the Gaussian elimination first. Ideas of this section are similar to those in \cite{plarre2004extended}, but the presentation is more straightforward and after appropriate modifications (see \cref{Section:Gaussian_belief_propagation.Section:GaBP_for_the_nonsymmetric_linear_system}), applies to non-symmetric matrices as well. 

To illustrate the main ideas, consider a linear problem $\mathbf{A}\mathbf{x} = \mathbf{b}$ with the matrix and right-hand side defined as
\begin{equation}\label{toy_GaBP}
    \mathbf{A} = 
    \begin{pmatrix}
        A_{11} & 0 & A_{13} & 0 & 0\\
        0 & A_{22} & A_{23} & 0 & 0\\
        A_{31} & A_{32} & A_{33} & 0 & A_{35}\\
        0 & 0 & 0 & A_{44} & A_{45}\\
        0 & 0 & A_{53} & A_{54} & A_{55}\\
    \end{pmatrix},~\mathbf{b}\in \mathbb{R}^{5}.
\end{equation}
For simplicity, require that $\mathbf{A}$ be positive definite and that all elements of $\mathbf{A}$, not explicitly indicated as zeros, are nonzero.

To obtain GaBP rules, we introduce a graph of the matrix~(\ref{toy_GaBP}). For this section, it is sufficient to associate the set of vertices with the set of diagonal terms and the collection of edges with nonzero entries $A_{ij},~i\neq j$. One can find the resulting graph in figure~\ref{fig:elimination}. The correspondence between graphs and matrices is discussed in detail in two subsequent subsections~\ref{Section:Gaussian_belief_propagation.Subsection:Conventional_belief_propagation}, \ref{Section:Gaussian_belief_propagation.Section:GaBP_for_the_nonsymmetric_linear_system}.

Suppose one wants to calculate  variable $x_5$. To do that, we exclude variables $x_1$, $x_2$ from the equation for $x_3$ and then eliminate variables $x_3$, $x_4$ from the equation for $x_5$
\begin{subequations}\label{elimination_5}
    \begin{align}
    &\underbrace{\left(A_{33} - \frac{A_{31}A_{13}}{A_{11}} - \frac{A_{32}A_{23}}{A_{22}}\right)}_{=\widetilde{A}_{33}} x_3 + A_{35} x_5 = \underbrace{b_3 - \frac{A_{31} b_1}{A_{11}} - \frac{A_{32} b_2}{A_{22}}}_{=\widetilde{b}_3} \label{modified_x5_intermediate};\\
    &\left(A_{55} - \frac{A_{53} A_{35}}{\widetilde{A}_{33}} - \frac{A_{54} A_{45}}{A_{44}}\right)x_5 = b_5 - \frac{A_{53}\widetilde{b}_3}{\widetilde{A}_{33}} - \frac{A_{54}b_4}{A_{44}}. \label{modified_x5}
    \end{align}
\end{subequations}
Figure~\ref{fig:elimination:a} captures this particular elimination order. In the same vein, to find $x_3$ one may follow the order presented in figure~\ref{fig:elimination:b}. The resulting equations are
\begin{subequations}\label{elimination_3}
    \begin{align}
    &\underbrace{\left(A_{55} - \frac{A_{54}A_{45}}{A_{44}}\right)}_{=\widetilde{A}_{55}}x_5 + A_{53}x_3 = \underbrace{b_5 - \frac{A_{53}b_4}{A_{44}}}_{=\widetilde{b}_{5}};\\
    &\left(A_{33} - \frac{A_{35}A_{53}}{\widetilde{A}_{55}} - \frac{A_{31}A_{13}}{A_{11}} - \frac{A_{32}A_{23}}{A_{22}}\right)x_3 = b_3 - \frac{A_{31}b_1}{A_{11}} - \frac{A_{32}b_2}{A_{22}} - \frac{A_{35}\widetilde{b}_5}{\widetilde{A}_{55}}. \label{modified_x3}
    \end{align}
\end{subequations}
From these calculations, one can make two observations:
\begin{enumerate}
    \item In the course of elimination one successively changes the diagonal elements $A_{jj}$ and the right-hand side $b_{j}$.
    \item The exclusion schemes in figures \ref{fig:elimination:a} and \ref{fig:elimination:b} share the same computations. For example, terms $A_{31}A_{13}\big/A_{11}$, $A_{32}A_{23}\big/A_{22}$ and $A_{54} A_{45}\big/A_{44}$ appear on the way to equation~(\ref{modified_x3}) as well as to (\ref{modified_x5}). It would be more advantageous to reuse the same computations, not to redo them each time one needs to eliminate a variable.
\end{enumerate}
The first observation suggests that one can introduce corrections to the diagonal terms and $b_{j}$, that come from the elimination of variable~$i$. For the sake of convenience, we denote them $\Lambda_{ji}$ and $\mu_{ji}\Lambda_{ji}$, respectively. For example, equation~(\ref{modified_x5_intermediate}) becomes
\begin{equation}\label{toy_messages}
    \left(A_{33} + \Lambda_{31} + \Lambda_{32}\right) x_3 + A_{35}x_{5} = b_{3} +  \Lambda_{32}\mu_{32} + \Lambda_{31}\mu_{31}.
\end{equation}
Since corrections are the same for any order of elimination, to reuse them, we can regard $\Lambda_{ji}$ and $\mu_{ji}$ as a message that node $i$ sends to node $j$ along the edge of the graph. Once computed, these messages are in use in expressions like (\ref{toy_messages}) and (\ref{toy_messages_2}). To complete rewriting the elimination in terms of messages, one needs to introduce the rules to update messages when a new variable is excluded. To derive the rules, we rewrite equation~(\ref{modified_x3}) using the definition of messages
\begin{equation}\label{toy_messages_2}
    \left(A_{55} + \Lambda_{53} + \Lambda_{54}\right) x_5 = b_5 + \Lambda_{53}\mu_{53} + \Lambda_{54}\mu_{54},
\end{equation}
and use (\ref{toy_messages}) to get
\begin{equation}\label{toy_update_rules}
    \Lambda_{53} = -\frac{A_{35} A_{53}}{A_{33} + \Lambda_{31} + \Lambda_{32}},~\mu_{53} = \frac{b_{3} + \Lambda_{31}\mu_{31} + \Lambda_{32}\mu_{32}}{A_{35}}.
\end{equation}
It is easy to see that one needs to accumulate all messages from neighbors of $i$ except for $j$ to send the message from node $i$ to node $j$. Since the update rule includes only messages from the previous stages of elimination, one can iterate equations like (\ref{toy_update_rules}) till convergence. And then, when all messages arrive, the solution can be read off as follows
\begin{equation}
    x_{j} = \frac{b_j + \sum\limits_{k\in\text{neighbours of }j} \Lambda_{jk}\mu_{jk}}{A_{jj} + \sum\limits_{k\in\text{neighbours of }j} \Lambda_{jk}}.
\end{equation}
Note that \eqref{elimination_3} and \eqref{elimination_5} have exactly this form. Equations for the update of messages that we deduced in this section coincide with the GaBP update rules given by (\ref{GaBP_rules}), which are  derived from the probabilistic perspective below.

To summarize, the GaBP rules can be understood as a scheme that propagates messages on the graph, corresponding to the matrix of the linear system under consideration. These messages, namely $\Lambda_{ji}$ and $\mu_{ji}$, represent the corrections to the  diagonal terms of matrix $\mathbf{A}$ and to the right-hand side $\mathbf{b}$, resulting from the elimination of variable $x_{i}$ from the $j$-th equation, $A_{jj} x_{j} + A_{ji} x_{i} + \dots = b_{j}$. Consistency, convergence, stopping criteria, and other practical matters are discussed in the following two sections.
\subsection{Conventional belief propagation}\label{Section:Gaussian_belief_propagation.Subsection:Conventional_belief_propagation}
Here we give a more traditional introduction to GaBP as a technique for statistical inference in graphical models. Following \cite{bickson2008gaussian,shental2008gaussian}, we reformulate \eqref{linear_problem} as an inference problem. For this purpose, consider a small subset of undirected graph models that are  known as Gauss-Markov random fields. First, we define a pairwise Markov random field. The graph $\Gamma$ is the set of edges $\mathcal{E}$ and vertices $\mathcal{V}$. Each vertex $i$ corresponds to the random variable $x_i$ (discrete or continuous), and each edge corresponds to interactions between variables. The set of non-negative integrable functions $\left\{\phi_i, \psi_{ij}\right\}$ together with the graph $\Gamma$ completely specifies the form of the probability density function of a pairwise Markov random field
\begin{equation}\label{Gibbs}
    p(x) = \frac{1}{Z}\prod_{i\in \mathcal{V}}\phi_i(x_i)\prod_{(ij)\in\mathcal{E}}\psi_{ij}(x_i, x_j)\equiv\frac{1}{Z}\exp\left(-E(\mathbf{x})\right).
\end{equation}
Here, $Z$ is a normalization constant (known in statistical physics as a partition function). The second equation in \eqref{Gibbs} (that is, the Gibbs distribution) should be considered as a definition of energy $E(\mathbf{x})$. The Gauss-Markov random field is a particular instance of a pairwise Markov model with a joint probability density function given by a multivariate normal distribution
\begin{equation}\label{Gauss_Markov_model}
    p(\mathbf{x}) = \frac{1}{Z} \exp\left(-\frac{\mathbf{x}^T \mathbf{A} \mathbf{x}}{2} + \mathbf{b}^T\mathbf{x}\right) \equiv \mathcal{N}\left(\mathbf{x}|\underset{\text{mean}}{\mathbf{A}^{-1}\mathbf{b}},\underset{\substack{\text{covariance} \\ \text{matrix}}}{\mathbf{A}^{-1}}\right),
\end{equation}
where $\mathbf{A}$ is a symmetric positive-definite matrix. The edges of $\Gamma$ correspond to the nonzero $A_{ij}$, and note that the splitting of the product in \eqref{Gibbs} into $\phi_i$ and  $\psi_{ij}$ is not unique. A common task in the inference process is a computation of a partial distribution (or a marginalization) 
\begin{equation}\label{marginalization}
    p_r(\mathbf{x}_r) = \sum_{\mathbf{x}\backslash \mathbf{x}_r}p(\mathbf{x}).
\end{equation}
Integrals replace sums if $\mathbf{x}\in\mathbb{R}^{N}$. For the Gauss-Markov model, marginal distributions are known explicitly. For individual components of the vector $\mathbf{x}$, which is distributed according to (\ref{Gauss_Markov_model}), one can obtain distributions in closed form 
\begin{equation}
    p_i(x_i) = \mathcal{N}\left(x_i|\left(\mathbf{A}^{-1}\mathbf{b}\right)_i, \left(\mathbf{A}^{-1}\right)_{ii}\right)\equiv \mathcal{N}\left(x_i|\mu_i, \beta_i\right).
\end{equation}
As the means of marginal distributions for the model (\ref{Gauss_Markov_model}) coincide with the elements of the solution vector for \eqref{linear_problem}, methods from the domain of probabilistic inference can be applied directly to obtain the solution. 

A popular algorithm that exploits the structure of the underlying graph to find the marginal distribution efficiently is Pearl's belief propagation \cite{pearl2014probabilistic}. Pearl's algorithm operates with local messages that spread from node to node along the graph edges, and beliefs (approximate or exact marginals) are computed as a normalized product of all incoming messages after the convergence. More precisely, belief propagation consists of (i) the message update rule
\begin{equation}
    m_{ij}(x_j)\leftarrow\sum\limits_{x_i}\phi_i(x_i)\psi_{ij}(x_i, x_j)\prod_{k\in N(i)\backslash j}m_{ki}(x_i),
\end{equation}
where $m_{ij}$ is a message from  node $i$ to node $j$ and $N(i)$ is the set of neighbors of the node $i$, and (ii) the formula for marginals
\begin{equation}
    b_i(x_i) \sim \phi_i(x_i)\prod\limits_{k\in N(i)} m_{ki}(x_i).
\end{equation}

Although, for continuous random variables the problem of marginalization and the algorithm of belief propagation are harder in general, it is not the case for the normal distribution. Namely, for the Gauss-Markov model, one can parameterize messages in the form of the normal distribution
\begin{equation}
    m_{ji}(x_i)\sim\exp\left(-\frac{\Lambda_{ji}\left(x_i-\mu_{ji}\right)^2}{2}\right),
\end{equation}
and explicitly derive update rules, means, and precision
\begin{equation}\label{GaBP_rules}
    \begin{split}
        &\mu_{ji}^{(n+1)} = \frac{b_j + \sum\limits_{k\in N(j)\backslash i} \Lambda_{kj}^{(n)} \mu_{kj}^{(n)}}{A_{ji}},~
        \Lambda_{ji}^{(n+1)} = -\frac{A_{ij}A_{ji}}{ A_{jj} + \sum\limits_{k\in N(j)\backslash i} \Lambda_{kj}^{(n)}};\\
        &\mu_i^{(n)} = \frac{b_i + \sum\limits_{j\in N(i)} \Lambda_{ji}^{(n)} \mu_{ji}^{(n)}}{ A_{ii} + \sum\limits_{j\in N(i)}\Lambda_{ji}^{(n)}},~
        \beta^{(n)}_i = A_{ii} + \sum_{j\in N(i)} \Lambda^{(n)}_{ji}.
    \end{split}
\end{equation}
These update rules correspond to the flood schedule such that at the current iteration step, each node sends messages to all its neighbours based on messages received at the previous step. Equations for the mean and precision should be put to use only after saturation according to some criteria, for example $|\mu^{(n+1)}-\mu^{(n)}|\leq \text{tolerance}$, and the same for $\mathbf{\Lambda}$. Rules (\ref{GaBP_rules}) are collectively known as Gaussian belief propagation.

Belief propagation was designed to give an exact answer if $\Gamma$ has no loops. In the presence of loops, the result appears to be approximate if delivered at all. In the case of GaBP, the situation is more optimistic. We briefly recall some useful facts about GaBP that we discuss later in more detail. If GaBP converges on the graph of arbitrary topology, the means are exact, but variances can be incorrect \cite{weiss2000correctness}. The best sufficient condition for convergence of the Gauss-Markov model with symmetric positive-definite matrix can be found in \cite{malioutov2006walk}, we discuss it later in greater detail. The fixed point of GaBP is unique \cite{heskes2004uniqueness}. On the tree,  GaBP is equivalent to the Gaussian elimination \cite{plarre2004extended}. 

Many different schemes that extend belief propagation and GaBP have been developed \cite{yedidia2003understanding}, \cite{el2012relaxed}, \cite{sudderth2004embedded}, \cite{minka2001expectation}, \cite{Elidan:2006:RBP:3020419.3020440}. Here, we are mainly interested in generalized belief propagation proposed in \cite{yedidia2003understanding} and subsequently developed in \cite{yedidia2001generalized}, \cite{yedidia2001bethe} ,\cite{yedidia2005constructing}. This new algorithm is significantly more accurate \cite[Fig. 15]{yedidia2005constructing} than Pearl's algorithm, but at the same time it can be computationally costly. In what follows, we show how to use the generalized belief propagation in the context of numerical linear algebra.

\subsection{GaBP for a nonsymmetric linear system}\label{Section:Gaussian_belief_propagation.Section:GaBP_for_the_nonsymmetric_linear_system}

As explained in \cref{Section:Gaussian_belief_propagation.Subsection:GaBP_from_the_elimination_perspective}, the GaBP rules can be understood as corrections to the right-hand side and the diagonal elements of the matrix under successive elimination of variables. It means that in principle, one can apply the rules to solve at least some nonsymmetric systems. However, there is a problem which is specific to nonsymmetric case. Namely, it is possible to have $A_{ij} = 0$ and $A_{ji} \neq 0$. In this case, rules (\ref{GaBP_rules}) lead to singularity as $A_{ij}$ appears in the denominator. Since parametrization of messages is not unique both from the elimination and probabilistic perspectives, it is possible to define new set of messages $\widetilde{\mathbf{\Lambda}}$ and $\mathbf{m}$ as follows
\begin{equation}\label{nonsymmetric_GaBP_messages}
    m_{ji}^{(n)} \equiv \mu_{ji}^{(n)}\Lambda_{ji}^{(n)},~\widetilde{\Lambda}_{ji}^{(n)} \equiv \Lambda_{ji}^{(n)}\big/A_{ji}.
\end{equation}
Then update rules become 
\begin{equation}\label{nonsymmetric_GaBP_rules}
    \begin{split}
        &m_{ji}^{(n+1)} = \widetilde{\Lambda}_{ji}^{(n+1)}\left(b_j + \sum\limits_{k\in N(j)\backslash i} m_{kj}^{(n)}\right),~\widetilde{\Lambda}_{ji}^{(n+1)} = -\frac{A_{ij}}{ A_{jj} + \sum\limits_{k\in N(j)\backslash i} \widetilde{\Lambda}_{kj}^{(n)} A_{kj}};\\
        &\mu_i^{(n)} = \frac{b_i + \sum\limits_{j\in N(i)} m_{ji}^{(n)}}{ A_{ii} + \sum\limits_{j\in N(i)}\widetilde{\Lambda}_{ji}^{(n)}A_{ji}},~\beta^{(n)}_i = A_{ii} + \sum_{j\in N(i)} \widetilde{\Lambda}^{(n)}_{ji}A_{ji}.
    \end{split}
\end{equation}

\begin{algorithm}[t]
    \caption{GaBP for a nonsymmetric linear system.}
    \label{algorithm:non_symmetric_GaBP}
    \begin{algorithmic}
        \STATE Form directed graph $G = \left\{\mathcal{V}, \mathcal{E}\right\}$ based on $\mathbf{A}$.
        \WHILE{$\text{error}>\text{tolerance}$} 
            \FOR{$j \in \mathcal{V}$}
                \STATE $m = b_j + \sum\limits_{k\in N(j)} m_{kj}$
                \STATE $\Sigma = A_{jj} + \sum\limits_{k\in N(j)} \widetilde{\Lambda}_{kj} A_{kj}$
                \STATE $x_{j} \leftarrow m/\Sigma$
                \FOR{$(j, i) \in \mathcal{E}$}
                    \STATE $\widetilde{\Lambda}_{ji}\leftarrow -A_{ij}/\left( \Sigma - \widetilde{\Lambda}_{ij} A_{ij}\right)$
                    \STATE $\widetilde{\mu}_{ji} \leftarrow \widetilde{\Lambda}_{ji}\left(m - m_{ij}\right)$
                \ENDFOR
            \ENDFOR
            \STATE $\text{error} = \left\|\mathbf{A}\mathbf{x}-\mathbf{b}\right\|_{\infty}$
        \ENDWHILE
    \end{algorithmic}
\end{algorithm}

Note that reparametrization (\ref{nonsymmetric_GaBP_messages}) has a problem in that it is not one-to-one if $A_{ji}=0$. However, the quick look at the equations (\ref{elimination_5}), (\ref{elimination_3}) makes clear that indeed it is possible to define messages in that way. That is, if $A_{jk}=0$, one does not need to eliminate $x_k$ from the second equation so the message $\Lambda_{kj}$ is indeed zero. 

For a given matrix $\mathbf{A}\in \mathbb{R}^{N\times N}$, we construct a graph with $N$ vertices $v\in\mathcal{V}$ corresponding to the variables $x_1, \dots, x_N$ and the set of directed edges $\mathcal{E}$. The edge pointing from the vertex $j$ to the vertex $i$ belongs to the set of edges iff $A_{ij}\neq 0$, i.e. $A_{ij}\neq0\Leftrightarrow e_{ji}\in\mathcal{E}$. This definition fixes the correspondence between directed graphs and nonsymmetric matrices and allows us to use GaBP (see \cref{algorithm:non_symmetric_GaBP}) beyond its usual domain of applicability.

\Cref{algorithm:non_symmetric_GaBP} is sequential, but can run in parallel after some modifications. The stopping criteria can be different, for example, it is possible to use different norms, or $\text{error} = \left\|\mathbf{x}^{(n+1)}-\mathbf{x}^{(n)}\right\|_{\infty}$, or \begin{equation}
    \text{error} = \max\left(\left\|\widetilde{\boldsymbol{\mu}}^{(n+1)}-\widetilde{\boldsymbol{\mu}}^{(n)}\right\|_{\infty}, \left\|\widetilde{\mathbf{\Lambda}}^{(n+1)}-\widetilde{\mathbf{\Lambda}}^{(n)}\right\|_{\infty}\right).
\end{equation}
Note that the update of $\widetilde{\mathbf{\Lambda}}$ decouples from the one for $\widetilde{\boldsymbol{\mu}}$. So it is possible to construct an algorithm that computes only messages $\widetilde{\boldsymbol{\Lambda}}$ and returns diagonal elements for the inverse matrix. Later, these messages can be used in the course of all successive iterations if one resorts to the error correction scheme.  We discuss how the algorithm of this kind can be utilized to decrease the number of floating point operations in the context of a  multigrid scheme.

One of the central results of the present work is that two classical theorems from GaBP theory, summarized below, can be readily established for nonsymmetric matrices.

\begin{restatable}{theorem}{GaBPconsistency}
\label{prop:GaBP_consistency}
    If there is $N\in\mathbb{N}$ such that $\widetilde{\mu}^{(N+k)}_{e} = \widetilde{\mu}^{(N)}_{e}$, $\widetilde{\Lambda}^{(N+k)}_{e} = \widetilde{\Lambda}^{(N)}_{e}$ for all $e\in\mathcal{E}$ and for any $k\in\mathbb{N}$, then $\mu^{(N+k)}_i = \mu^{(N)}_i = \left(\mathbf{A}^{-1}\mathbf{b}\right)_i$.
\end{restatable}
The analogous result for symmetric matrices first appeared in \cite{weiss2000correctness}. In \ref{Appendices:Consistency_of_GaBP}, we show how to extend the proof for the nonsymmetric case.

\begin{restatable}{theorem}{GaBPconvergence}\label{prop:GaBP_convergence}
If $A_{ii}\neq 0~\forall i$, $\widetilde{|R|}_{ij} = \left(1 - \delta_{ij}\right)\frac{\left|A_{ij}\right|}{\left|A_{ii}\right|}$, and $\rho(\widetilde{|\mathbf{R}|})<1$, then the \cref{algorithm:non_symmetric_GaBP} converges to the solution $\mathbf{x}^{\star} = \mathbf{A}^{-1}\mathbf{b}$ for any $\mathbf{b}$.
\end{restatable}

Sufficient condition for symmetric positive-definite matrices was established in \cite{malioutov2006walk}. \Cref{Appendices:Convergence_of_GaBP} contains the proof with necessary modifications that holds for nonsymmetric matrices.

To make connections with the classical theory of iterative methods, we give another (less general) sufficient condition.
\begin{corollary}\label{prop:M_matrices}
If $\mathbf{A}$ is the $M$-matrix (see \cite[Definition 1.30, Theorem 1.31]{saad2003iterative}), \cref{algorithm:non_symmetric_GaBP} converges to the solution $\mathbf{x}^{\star} = \mathbf{A}^{-1}\mathbf{b}$ for any $\mathbf{b}$.  
\end{corollary}
\begin{proof}
For $M$-matrix $\rho(\mathbf{I}-\mathbf{D}^{-1}\mathbf{A})<1$, $A_{ij}\leq 0,~i\neq j$ and $A_{ii}>0$, where $\mathbf{D}$ is a diagonal of $\mathbf{A}$. It means that $\mathbf{\widetilde{R}} = \mathbf{I}-\mathbf{D}^{-1}\mathbf{A} = \left|\mathbf{\widetilde{R}}\right|$ and $\rho\left(\left|\mathbf{\widetilde{R}}\right|\right)<1$.
\end{proof}

\section{Generalized GaBP solvers}\label{Section:Generalized_GaBP_solvers}
\begin{figure}
    \centering
    \subfloat[]{\label{fig:region_graph:a}\begin{tikzpicture}[
sq/.style={circle, draw=black, minimum size=1mm},
]
\node[sq] at (0, 0) (1) {$1$};
\node[sq] at (1, 0) (2) {$2$};
\node[sq] at (2, 0) (3) {$3$};
\node[sq] at (0, -1) (4) {$4$};
\node[sq] at (1, -1) (5) {$5$};
\node[sq] at (2, -1) (6) {$6$};
\node[sq] at (0, -2) (7) {$7$};
\node[sq] at (1, -2) (8) {$8$};
\node[sq] at (2, -2) (9) {$9$};
\draw[-] (1) -- (2);
\draw[-] (1) -- (4);
\draw[-] (2) -- (3);
\draw[-] (2) -- (5);
\draw[-] (3) -- (6);
\draw[-] (4) -- (5);
\draw[-] (4) -- (7);
\draw[-] (5) -- (6);
\draw[-] (5) -- (8);
\draw[-] (6) -- (9);
\draw[-] (7) -- (8);
\draw[-] (8) -- (9);
\end{tikzpicture}}\quad
    \subfloat[]{\label{fig:region_graph:b}\begin{tikzpicture}[
sq/.style={rectangle, draw=black, minimum size=1mm},
]
\node[sq, label=west:+1, very thick] at (-2, 0) (235689) {$235689$};
\node[sq, label=west:+1] at (1, 0) (123456) {$123456$};
\node[sq, label=west:0, fill={rgb:black,1;white,2}] at (-3, -1) (5689) {$5689$};
\node[sq, label=west:-1, very thick] at (-1, -1) (2356) {$2356$};
\node[sq, label=west:+1, very thick] at (1, -1) (4578) {$4578$};
\node[sq, label=west:0, fill={rgb:black,1;white,2}] at (-4.5, -2) (89) {$89$};
\node[sq, label=west:0, fill={rgb:black,1;white,2}] at (-3, -2) (69) {$69$};
\node[sq, label=west:0, fill={rgb:black,1;white,2}] at (-1.5, -2) (56) {$56$};
\node[sq, label=west:-1, fill={rgb:black,1;white,2}] at (0.5, -2) (58) {$58$};
\node[sq, label=west:-1, very thick] at (2.5, -2) (45) {$45$};
\node[sq, label=west:0, fill={rgb:black,1;white,2}] at (-1.5, -3) (6) {$6$};
\node[sq, label=west:+1, fill={rgb:black,1;white,2}] at (0.5, -3) (5) {$5$};
\draw[->] (235689) -- (5689);
\draw[->] (235689) -- (2356);
\draw[->] (123456) -- (2356);
\draw[->] (5689) -- (89);
\draw[->] (5689) -- (69);
\draw[->] (5689) -- (56);
\draw[->] (5689) -- (58);
\draw[->] (2356) -- (56);
\draw[->] (4578) -- (58);
\draw[->] (4578) -- (45);
\draw[->] (123456) -- (45);
\draw[->] (56) -- (6);
\draw[->] (56) -- (5);
\draw[->] (58) -- (5);
\draw[->] (45) -- (5);
\end{tikzpicture}}
    \caption{Pairwise Markov model and valid region graph with counting numbers. Shaded nodes belong to the shadow of $5689$, and nodes with thick borders to the Markov blanket of $5689$. See \cref{Section:Generalized_GaBP_solvers.Subsection:Parent-to-child_algorithm} for details.}
    \label{fig:region_graph}
\end{figure}
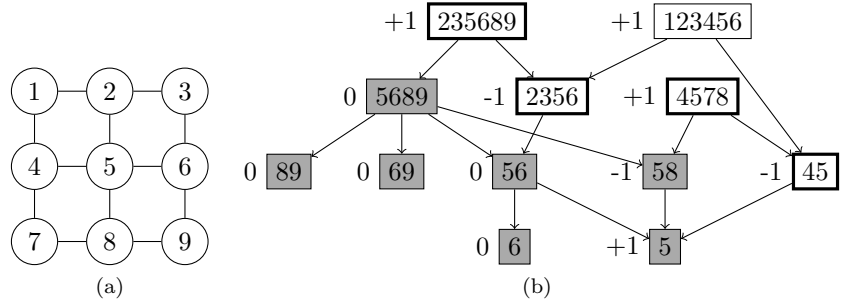
\subsection{Parent-to-child schedule}\label{Section:Generalized_GaBP_solvers.Subsection:Parent-to-child_algorithm}
We present a particular version of the parent-to-child schedule from \cite{yedidia2005constructing} applied to the pairwise Markov graphical models. 

First, we define a region $r$ as a connected subgraph of the original graph $\Gamma$. Each vertex $a_1$ in the region graph is a region that may be connected by a directed edge with another vertex  $a_2$ if $a_2\subset a_1$. The direction of the edge is from a larger region to smaller. If there is a directed edge from $a$ to $b$, then $a$ is a parent of $b$, and $b$ is a child of $a$. If the vertices $a$ and $b$ are connected by a directed path starting from $a$, then $a$ is an ancestor of $b$, and $b$ is a descendent of $a$. The set of all vertices of the factor graph is $\mathcal{R}$, the set of all edges is $\mathcal{E}_{\mathcal{R}}$, the set of all parents, children, ancestors, descendants of $a$ are $P(a)$, $C(a)$, $A(a)$, and $D(a)$, respectively. We supplement each region $r\in\mathcal{R}$ with a counting number,
\begin{equation}\label{counting_definition}
    c_r = 1 - \sum_{i\in A(r)} c_i,
\end{equation}
and require that each vertex $v\in\mathcal{V}$ and each edge $e\in\mathcal{E}$ of the original graph be counted exactly once,
\begin{equation}\label{counting_condition}
    \sum_{r\in\mathcal{R}, v\in r} c_r =  \sum_{r\in\mathcal{R}, e\in r} c_r = 1.
\end{equation}
The definitions of a counting number (\ref{counting_definition}) and  condition (\ref{counting_condition}) are justified in the framework of the cluster variation method \cite{pelizzola2005cluster} (or the Kikuchi approximation \cite{kikuchi1951theory}). Equation (\ref{counting_condition}) is a result of the Möbius inversion formula applied to the sum over partially ordered sets \cite{an1988note}, and (\ref{counting_condition}) can be proven using definitions of Möbius and Zeta functions \cite[equation 16]{an1988note}.

A sample region graph is shown in \cref{fig:region_graph}. For example, by region $5689$, we mean all nodes and all links between them that are present on the original graph. It is easy to see that the counting number condition (\ref{counting_condition}) is satisfied for all the links and nodes.

The last two definitions that we need are the shadow of the region $S(r) = D(r)\cup r$ and a Markov blanket of the region $B(r) = P\left(S(r)\right)\backslash S(r)$. An example of both the shadow and Markov blanket of region $5689$ can be found in \cref{fig:region_graph:b}. 

A parent-to-child algorithm consists of three elements: (i) messages that propagate along the directed edges of the region graph
\begin{equation}\label{generalized_message}
    \mathbf{m}_{a\rightarrow b}(\mathbf{x}_b),
\end{equation}
where $x_b$ corresponds to variables belong to the cluster $b$, (ii) a formula for the cluster beliefs
\begin{equation}\label{generalized_marginals}
    \mathbf{b}_r(\mathbf{x}_r)\sim \prod_{i\in \mathcal{V}_r} \phi_i(x_i) \prod_{(ij)\in \mathcal{E}_r} \psi_{ij}(x_i, x_j) \prod_{a\in B(r),~b\in S(r)} \mathbf{m}_{a\rightarrow b}(\mathbf{x}_b), 
\end{equation}
and (iii) message update rules that follow from the consistency conditions
\begin{equation}\label{consistency}
    \forall l, r\in \mathcal{R},~l\subset r \Rightarrow \sum_{x_r\backslash x_l} \mathbf{b}_{r}(\mathbf{x}_r) = \mathbf{b}_{l}(\mathbf{x}_l).
\end{equation}

In \cite[equation 114]{yedidia2005constructing}, one can find message update rules for the general situation, but for our simple region graph, conditions (\ref{consistency}) suffice.

\subsection{Two-layer generalized GaBP}\label{Section:Generalized_GaBP_solvers.Subsection:Generalized_two_layers_GaBP}
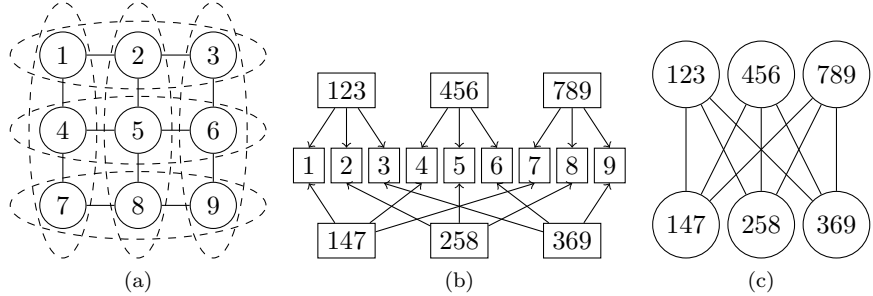
\begin{figure}
    \centering
    \subfloat[]{\label{fig:two_layers_network:a}\begin{tikzpicture}[
sq/.style={circle, draw=black, minimum size=1mm},
]
\node[sq] at (0, 0) (1) {$1$};
\node[sq] at (1, 0) (2) {$2$};
\node[sq] at (2, 0) (3) {$3$};
\node[sq] at (0, -1) (4) {$4$};
\node[sq] at (1, -1) (5) {$5$};
\node[sq] at (2, -1) (6) {$6$};
\node[sq] at (0, -2) (7) {$7$};
\node[sq] at (1, -2) (8) {$8$};
\node[sq] at (2, -2) (9) {$9$};
\draw[-] (1) -- (2);
\draw[-] (1) -- (4);
\draw[-] (2) -- (3);
\draw[-] (2) -- (5);
\draw[-] (3) -- (6);
\draw[-] (4) -- (5);
\draw[-] (4) -- (7);
\draw[-] (5) -- (6);
\draw[-] (5) -- (8);
\draw[-] (6) -- (9);
\draw[-] (7) -- (8);
\draw[-] (8) -- (9);
\draw[dashed] (0,-1) circle [x radius=0.45cm, y radius=1.7cm];
\draw[dashed] (1,-1) circle [x radius=0.45cm, y radius=1.7cm];
\draw[dashed] (2,-1) circle [x radius=0.45cm, y radius=1.7cm];
\draw[dashed] (1,0) circle [x radius=1.7cm, y radius=0.45cm];
\draw[dashed] (1,-1) circle [x radius=1.7cm, y radius=0.45cm];
\draw[dashed] (1,-2) circle [x radius=1.7cm, y radius=0.45cm];
\end{tikzpicture}}\quad
    \subfloat[]{\label{fig:two_layers_network:b}\begin{tikzpicture}[
sq/.style={rectangle, draw=black, minimum size=1mm},
]
\node[sq] at (-0.5, 0) (H1) {$123$};
\node[sq] at (1, 0) (H2) {$456$};
\node[sq] at (2.5, 0) (H3) {$789$};
\node[sq] at (-1, -1) (1) {$1$};
\node[sq] at (-0.5, -1) (2) {$2$};
\node[sq] at (0, -1) (3) {$3$};
\node[sq] at (0.5, -1) (4) {$4$};
\node[sq] at (1, -1) (5) {$5$};
\node[sq] at (1.5, -1) (6) {$6$};
\node[sq] at (2, -1) (7) {$7$};
\node[sq] at (2.5, -1) (8) {$8$};
\node[sq] at (3, -1) (9) {$9$};
\node[sq] at (-0.5, -2) (V1) {$147$};
\node[sq] at (1, -2) (V2) {$258$};
\node[sq] at (2.5, -2) (V3) {$369$};
\draw[->] (H1) -- (1.north);
\draw[->] (H1) -- (2.north);
\draw[->] (H1) -- (3.north);
\draw[->] (H2) -- (4.north);
\draw[->] (H2) -- (5.north);
\draw[->] (H2) -- (6.north);
\draw[->] (H3) -- (7.north);
\draw[->] (H3) -- (8.north);
\draw[->] (H3) -- (9.north);
\draw[->] (V1) -- (1.south);
\draw[->] (V1) -- (4.south);
\draw[->] (V1) -- (7.south);
\draw[->] (V2) -- (2.south);
\draw[->] (V2) -- (5.south);
\draw[->] (V2) -- (8.south);
\draw[->] (V3) -- (3.south);
\draw[->] (V3) -- (6.south);
\draw[->] (V3) -- (9.south);
\end{tikzpicture}}\quad
    \subfloat[]{\label{fig:two_layers_network:c}\begin{tikzpicture}[
sq/.style={circle, draw=black, minimum size=1mm},
]
\node[sq] at (0, 0) (H1) {$123$};
\node[sq] at (1, 0) (H2) {$456$};
\node[sq] at (2, 0) (H3) {$789$};
\node[sq] at (0, -2) (V1) {$147$};
\node[sq] at (1, -2) (V2) {$258$};
\node[sq] at (2, -2) (V3) {$369$};
\draw[-] (H1) -- (V1);
\draw[-] (H1) -- (V2);
\draw[-] (H1) -- (V3);
\draw[-] (H2) -- (V1);
\draw[-] (H2) -- (V2);
\draw[-] (H2) -- (V3);
\draw[-] (H3) -- (V1);
\draw[-] (H3) -- (V2);
\draw[-] (H3) -- (V3);
\end{tikzpicture}}
    \caption{An example of a region graph. (a) - the original graph partitioning, (b) - the region graph, and (c) - the hypergraph structure. In \cref{Section:Numerical_examples}, this partitioning is referred to as ``line GaBP''.}
    \label{fig:two_layers_network}
\end{figure}
In this section, we consider the simplest possible valid region graph that consists of two layers, \cref{fig:two_layers_network:b}. The large regions, the horizontal and vertical stripes in \ref{fig:two_layers_network:a}, are parents of small regions presented by the individual nodes. To proceed, we need to establish some further notation. First, we define a projector on the region $k$,
\begin{equation}\label{projector}
    \left(\Pi_{k}\right)_{ij}=
    \begin{cases}
        \delta_{ij},~i, j \in k,\\
        0,~\text{otherwise}.\\
    \end{cases}
\end{equation}
Here, $|j| = |k|$ and $|i|>|k|$ is chosen to be conformable depending on the context. Ordering is global, i.e., it is fixed for the whole graph and maintained the same way in all manipulations. We also introduce brackets,
\begin{subequations}
    \begin{align}
        &\mathbf{\Pi}_k^T \mathbf{A} \mathbf{\Pi}_k = \left[\mathbf{A}\right]_k, \mathbf{\Pi}_k^T \mathbf{b} = \left[\mathbf{b}\right]_{k} \label{fan_in};\\
        &\mathbf{\Pi}_k \mathbf{C} \mathbf{\Pi}_k^T = \left]\mathbf{C}\right[_k, \mathbf{\Pi}_k \mathbf{b} = \left]\mathbf{b}\right[_{k} \label{fan_off},
    \end{align}
\end{subequations}
where $\mathbf{C}\in\mathbb{R}^{|k|\times |k|}$ and according to our notation, the sizes of the matrix and the vector (\ref{fan_off}) depend of the context whereas in (\ref{fan_in}) the size of the matrix is $|k|\times |k|$ and the size of the vector is $|k|$.

Let $\left\{L_i\right\}$ and $\left\{l_i\right\}$ be sets of large and small regions in the two-layer region graph. For messages, we use the following parameterizations
\begin{equation}\label{message_parametrization}
    m_{ab}(\mathbf{x}_b)=\mathcal{N}\left(\mathbf{x}_b\left|\boldsymbol{\mu}_{ab}, \mathbf{\Lambda^{-1}}_{ab}\right.\right) = \mathcal{N}\left(\mathbf{x}_b\left|\boldsymbol{\mu}_{ab}, \mathbf{\Sigma}_{ab}\right.\right)
\end{equation}
According to (\ref{generalized_marginals}), the belief of any region reads
\begin{subequations}
    \begin{align}
        &\mathbf{b}_{L}(\mathbf{x}_L) = \mathcal{N}\left(\mathbf{x}_{L}\left|\boldsymbol{\Sigma}_{L}\mathbf{m}_{L}, \boldsymbol{\Sigma}_{L}\right.\right),\\
        &\mathbf{m}_{L} = \left[\mathbf{b}\right]_{L} + \sum\limits_{\substack{a\in\mathcal{B}(L)\\b\in\mathcal{S}(L)}}\left]\boldsymbol{\Lambda}_{ab}\boldsymbol{\mu}_{ab}\right[_{b},\\
        &\boldsymbol{\Sigma}_{L} = \left(\left[\mathbf{A}\right]_L + \sum\limits_{\substack{a\in\mathcal{B}(L)\\b\in\mathcal{S}(L)}}\left]\boldsymbol{\Lambda}_{ab}\right[_{b}\right)^{-1}.
    \end{align}
\end{subequations}
For the small region $l$, the shadow is $\mathcal{S}(l) = l$ and the Markov blanket is $\mathcal{B}(l) = P(l)$, whereas for the large region $L$, the shadow is $\mathcal{S}(L) = C(L)$ and the Markov blanket is $\mathcal{B}(L)=P(C(L))\backslash L$. Consistency condition (\ref{consistency}) allows one to derive update rules for each message sent from the parent region to the child region,
\begin{subequations}\label{region_graph_messages}
    \begin{align}
        &\mathbf{m}_{Ll}(x_{l}) = \int \frac{\mathbf{dx}_L}{\mathbf{dx}_l} \mathcal{N}\left(\mathbf{x}_{l}\left|\widetilde{\boldsymbol{\Sigma}}\widetilde{\boldsymbol{b}}, \widetilde{\boldsymbol{\Sigma}}\right.\right),\\
        &\widetilde{\mathbf{b}} = \left[\mathbf{b}\right]_{L} - \left]\left[\mathbf{b}\right]_{l}\right[_{l} + \sum\limits_{\substack{a\in\mathcal{B}(L)\\b\in\mathcal{S}(L)\backslash l}}\left]\boldsymbol{\Lambda}_{ab}\boldsymbol{\mu}_{ab}\right[_{b},\\
        &\widetilde{\boldsymbol{\Sigma}} = \left(\left[\mathbf{A}\right]_L - \left]\left[\mathbf{A}\right]_{l}\right[_{l} + \sum\limits_{\substack{a\in\mathcal{B}(L)\\b\in\mathcal{S}(L)\backslash l}}\left]\boldsymbol{\Lambda}_{ab}\right[_{b}\right)^{-1}. \label{modified_covariance}
    \end{align}
\end{subequations}
Using standard results for marginals of the normal distribution, we can deduce
\begin{equation}\label{rough_message_update}
    \boldsymbol{\Lambda}_{Ll} = \left(\left[\widetilde{\boldsymbol{\Sigma}}\right]_{l}\right)^{-1},~\boldsymbol{\mu}_{Ll} = \left[\widetilde{\boldsymbol{\Sigma}}\widetilde{\boldsymbol{b}}\right]_{l}.
\end{equation}
If region $L$ has many children and $|l|\ll |L|$, the direct use of equation (\ref{modified_covariance}) is not efficient. Instead, we use the formula following from the Woodbury matrix identity,
\begin{subequations}
    \begin{align}
        &\left[\left(\mathbf{A} + \left]\left[\mathbf{B}\right]_{l}\right[_{l}\right)^{-1}\right]_{l} = \left(\left(\left[\mathbf{A}^{-1}\right]_{l}\right)^{-1} + \left[\mathbf{B}\right]_{l}\right)^{-1},\\
        \begin{split}
            &\left[\left(\mathbf{A} + \left]\left[\mathbf{B}\right]_{l}\right[_{l}\right)^{-1}\left(\mathbf{b} + \left]\left[\mathbf{c}\right]_l\right[_l\right)\right]_{l} = \\ 
            &= \left(\left(\left[\mathbf{A}^{-1}\right]_{l}\right)^{-1} + \left[\mathbf{B}\right]_{l}\right)^{-1} \left(\left(\left[\mathbf{A}^{-1}\right]_{l}\right)^{-1} \left[\mathbf{A}^{-1}\mathbf{b}\right]_l + \left[\mathbf{c}\right]_l\right),
        \end{split}
    \end{align}
\end{subequations}
and split for each child region $l$,
\begin{subequations}
    \begin{align}
        &\widetilde{\mathbf{b}} = \left(\left[\mathbf{b}\right]_{L} + \sum\limits_{\substack{a\in\mathcal{B}(L)\\b\in\mathcal{S}(L)}}\left]\boldsymbol{\Lambda}_{ab}\boldsymbol{\mu}_{ab}\right[_{b}\right) - \left(\left]\left[\mathbf{b}\right]_{l}\right[_{l} + \sum\limits_{a\in P(l)\backslash L}\left]\boldsymbol{\Lambda}_{al}\boldsymbol{\mu}_{al}\right[_{l}\right) = \widetilde{\mathbf{b}}_0 - \widetilde{\mathbf{b}}_l,\\
        &\widetilde{\boldsymbol{\Sigma}}^{-1} = \left(\left[\mathbf{A}\right]_L + \sum\limits_{\substack{a\in\mathcal{B}(L)\\b\in\mathcal{S}(L)}}\left]\boldsymbol{\Lambda}_{ab}\right[_{b}\right) - \left(\left]\left[\mathbf{A}\right]_{l}\right[_{l} + \sum\limits_{a\in P(l)\backslash L}\left]\boldsymbol{\Lambda}_{al}\right[_{l}\right) = \widetilde{\boldsymbol{\Lambda}}_0 - \widetilde{\boldsymbol{\Lambda}}_l.
    \end{align}
\end{subequations}
Then, for precision and mean of messages, we obtain
\begin{subequations}\label{Two_layers_GaBP_rules}
    \begin{align}
        &\boldsymbol{\Lambda}_{Ll} = \left(\left[\widetilde{\boldsymbol{\Lambda}}_0^{-1}\right]_{l}\right)^{-1} - \widetilde{\boldsymbol{\Lambda}}_l,\\
        &\boldsymbol{\mu}_{Ll} = \boldsymbol{\Lambda}_{Ll}^{-1}\left(\left(\left[\widetilde{\boldsymbol{\Lambda}}_0^{-1}\right]_{l}\right)^{-1}\left[\widetilde{\boldsymbol{\Lambda}}_0^{-1} \widetilde{\mathbf{b}}_0\right]_l - \widetilde{\mathbf{b}}_l\right).
    \end{align}
\end{subequations}
Equations (\ref{Two_layers_GaBP_rules}) are especially useful in the situation when the small regions consist of the single node, i.e. the case of graph in \cref{fig:two_layers_network}. In this situation, one needs to invert the matrix corresponding to the large region only once whereas the direct application of (\ref{rough_message_update}) leads to  $|C(a)|$ inversions.

From formulae (\ref{Two_layers_GaBP_rules}), one can derive GaBP rules. As the large regions in the Bethe approximation consist of two vertices with the edge connecting them, we can rewrite messages as
\begin{equation}
    \Lambda_{(ij)i}\equiv \Lambda_{ji},~\mu_{(ij)i}\equiv \mu_{ji}.
\end{equation}
Then, the Gaussian belief propagation rules follow from 
\begin{subequations}
    \begin{align}
        &\widetilde{\mathbf{\Lambda}}_0 = 
        \begin{pmatrix}
            A_{ii} + \sum\limits_{k\in N(i)\backslash j} \Lambda_{ki}&A_{ij}\\
            A_{ji}&A_{jj} + \sum\limits_{k\in N(j)\backslash i} \Lambda_{kj}\\
        \end{pmatrix},\\
        &\widetilde{\Lambda}_i = A_{ii} + \sum\limits_{k\in N(i)\backslash j} \Lambda_{ki},\\
        &\left]\widetilde{\Lambda}_i\right[_i = 
        \begin{pmatrix}
            A_{ii} + \sum\limits_{k\in N(i)\backslash j} \Lambda_{ki}&0\\
            0&0\\
        \end{pmatrix}.
    \end{align}
\end{subequations}

The validity of the presented rules for nonsymmetric linear problems does not follow from the derivation above. Nevertheless, one  can apply the generalized GaBP rules with no modifications to solve them, as we explain in the next sections.

\subsection{Elimination perspective}\label{Section:Generalized_GaBP_solvers.Subsection:Elimination_perspective}
First, we define a hypergraph $\mathcal{G}$ based on the region graph. The set of nodes coincides with the set of large regions, and each common child corresponds to the edge in the graph. The example of such a hypergraph is shown in \cref{fig:two_layers_network:c}. With this definition, messages from equations (\ref{region_graph_messages}) can be redefined with no reference to the small region as long as a one-to-one correspondence between edges of $\mathcal{G}$ and child regions of the original region graph are established. Now we study the single message from  region $j$ to  region $i$ with $k = j\cap i$ and $\overline{j} = j\backslash k$. According to (\ref{region_graph_messages}) and (\ref{rough_message_update}), the precision part of the message is 
\begin{equation}
    \boldsymbol{\Lambda}_{ji} = \left(\left(\begin{pmatrix}
        \mathbf{A}_{\overline{j}\overline{j}} & \mathbf{A}_{\overline{j}k}\\
        \mathbf{A}_{k\overline{j}} & 0\\
    \end{pmatrix}^{-1}\right)_{kk}\right)^{-1} = - \mathbf{A}_{k\overline{j}}\mathbf{A}_{\overline{j}\overline{j}}^{-1} \mathbf{A}_{\overline{j}k},
\end{equation}
and the mean is
\begin{equation}
    \boldsymbol{\mu}_{ji} = \left(\begin{pmatrix}
        \mathbf{A}_{\overline{j}\overline{j}} & \mathbf{A}_{\overline{j}k}\\
        \mathbf{A}_{k\overline{j}} & 0\\
    \end{pmatrix}^{-1}
    \begin{pmatrix}
        \mathbf{b}_{\overline{j}}\\
        0\\
    \end{pmatrix}\right)_k = -\boldsymbol{\Lambda}_{ji}^{-1} \mathbf{A}_{k\overline{j}}\mathbf{A}_{\overline{j}\overline{j}}^{-1} \mathbf{b}_{\overline{j}}.
\end{equation}
So the subset $k$ of $\mathbf{A}_{i}$ and $\mathbf{b}_{i}$ receive a correction from region $j$,
\begin{subequations}\label{single_region_corrections}
    \begin{align}
        &\Delta \left(\mathbf{A}_i\right)_{kk} = -\mathbf{A}_{k\overline{j}}\mathbf{A}_{\overline{j} \overline{j}}^{-1} \mathbf{A}_{\overline{j}k},\\
        &\Delta \left(\mathbf{b}_i\right)_{k} = - \mathbf{A}_{k\overline{j}}\mathbf{A}_{\overline{j}\overline{j}}^{-1} \mathbf{b}_{\overline{j}}.
    \end{align}
\end{subequations}
The corrections above are the same as in the ordinary block LU decomposition,
\begin{equation}\label{LU}
    \begin{pmatrix}
        \mathbf{A}_{\overline{j}\overline{j}} & \mathbf{A}_{\overline{j}k}\\
        \mathbf{A}_{k\overline{j}} & \mathbf{A}_{kk}\\
    \end{pmatrix} = 
    \begin{pmatrix}
        \mathbf{I} & 0\\
        \mathbf{A}_{k\overline{j}}\mathbf{A}_{\overline{j}\overline{j}}^{-1} & \mathbf{I}\\
    \end{pmatrix}
    \begin{pmatrix}
        \mathbf{A}_{\overline{j}\overline{j}} & \mathbf{A}_{\overline{j} k}\\
        0 & \mathbf{A}_{kk} - \mathbf{A}_{k\overline{j}}\mathbf{A}_{\overline{j} \overline{j}}^{-1} \mathbf{A}_{\overline{j}k}.\\
    \end{pmatrix}
\end{equation}
Thus, we have a correspondence between the block LU operations and the generalized GaBP for the two-layer region graph. As in the case of the regular GaBP, LU applies in the fashion of dynamic programming, i.e., one does not just solve a smaller subproblem as in the block iterative scheme, but rather forms a recursive procedure that decouples different blocks from each other.
Again, one should reparametrize messages for them to be valid for the arbitrary invertible matrix $\mathbf{A}$. Note that if $\mathbf{A}_{jk} = 0$, then $\mathbf{\Lambda}_{ji}$ is not invertible. However, this is not a problem because in this case $\mathbf{A}_i$ does not receive corrections $\Delta \left(\mathbf{\mathbf{A}}_i\right)_{kk}$, and the reparametrization for the mean messages $\mathbf{m}_{Ll} \equiv \mathbf{\Lambda}_{Ll}\boldsymbol{\mu}_{Ll}$ resolves the issue.

\subsection{The algorithm}\label{Section:Generalized_GaBP_solvers.Subsection:The_algorithm}
In the case of generalized GaBP, messages propagate on the region graph. To remind, $\left\{L_i\right\}$ is the set of large regions, and $\left\{l_i\right\}$ is the set of small regions.
\begin{algorithm}[t]
    \caption{Generalized two-layer GaBP for a nonsymmetric linear system.}
    \label{algorithm:non_symmetric_two_layers_GaBP}
    \begin{algorithmic}
        \STATE For a given two-layer region graph $G = \left\{\mathcal{R}, \mathcal{E}_{\mathcal{R}}\right\}$.
        \WHILE{$\text{error}>\text{tolerance}$} 
            \FOR{$L \in \left\{L_i\right\}\subset \mathcal{R}$}
                \STATE $\widetilde{\mathbf{b}}_0 = \left[\mathbf{b}\right]_L$
                \STATE $\widetilde{\mathbf{\Lambda}}_0 = \left[\mathbf{A}\right]_{L}$
                \FOR{$l \in C(L)$}
                    \STATE $\widetilde{\mathbf{b}}_0 \leftarrow \widetilde{\mathbf{b}}_0 + \sum\limits_{L^{'}\in P(C(l))\backslash L}\left]\mathbf{m}_{L^{'}l}\right[_{l}$
                    \STATE $\widetilde{\mathbf{\Lambda}}_0 \leftarrow \widetilde{\mathbf{\Lambda}}_0 + \sum\limits_{L^{'}\in P(C(l))\backslash L}\left]\mathbf{\Lambda}_{L^{'}l}\right[_{l}$
                \ENDFOR
                \STATE $\left[\mathbf{x}\right]_{L} \leftarrow \widetilde{\mathbf{\Lambda}}_0^{-1} \widetilde{\mathbf{b}}_0$
                \FOR{$l \in C(L)$}
                    \STATE $\mathbf{m}_{Ll} \leftarrow \left(\left[\widetilde{\mathbf{\Lambda}}_0^{-1}\right]_{l}\right)^{-1}\left[\mathbf{x}\right]_l - \left]\left[\mathbf{b}\right]_l\right[_l - \sum\limits_{L^{'}\in P(C(l))\backslash L}\mathbf{m}_{L^{'}l}$
                    \STATE $\mathbf{\Lambda}_{Ll}\leftarrow \left(\left[\widetilde{\mathbf{\Lambda}}_0^{-1}\right]_{l}\right)^{-1} - \left]\left[\mathbf{A}\right]_l\right[_l -\sum\limits_{L^{'}\in P(C(l))\backslash L}\mathbf{\Lambda}_{L^{'}l}$
                \ENDFOR
            \ENDFOR
            \STATE $\text{error} = \left\|\mathbf{A}\mathbf{x}-\mathbf{b}\right\|_{\infty}$
        \ENDWHILE
    \end{algorithmic}
\end{algorithm}The algorithm, as we describe it in this section, acts on a given region graph. We do not provide an algorithm or recommendations on how to build a region graph. Some observations about the influence of a particular choice of the large regions can be found in \cref{Section:Numerical_examples}. Regarding complexity, for each region, one needs to solve a linear system of $\left|L\right|$ equations, and also find a submatrix of the inverse matrix of size $\left|L\right|\times \left|L\right|$. We do not specify how to do it. But generally, the former task is hard to accomplish asymptotically faster than the whole inverse, so the computational cost of the entire scheme depends mostly on this operation. For this reason, it what follows we mostly consider region graphs with small regions each of which contains a single node. In this situation, one can avoid fill-in and estimate only the diagonal of the inverse matrix which is usually more straightforward. Also, there is no need for an additional inverse step during the message update. Moreover, when all small regions are single nodes, the whole algorithm is either a method that speeds up GaBP or a particular schedule of GaBP depending on the chosen way of finding the inverse and the solution of a linear system.

To perform a worst-case analysis both of \cref{algorithm:non_symmetric_two_layers_GaBP} and update rules (\ref{Two_layers_GaBP_rules}), for a given region with $N$ variables, we denote the number of children by $M$, the number of variables of each child region by $n_i$, the number of parents for each child by $p_i$, and use the LU decomposition to find an inverse. For the \cref{algorithm:non_symmetric_two_layers_GaBP}, the number of operations is
\begin{equation}
    \#_1 =\sum_{i=1}^{M}\left[ \frac{3}{2} n_i^3 + 2n_i \left(n_i + 1\right)\left(p_i - 1\right)\right] + \frac{3}{2}N^3.
\end{equation}
The first term in brackets is due to the inverse during the message update stage, the second term in brackets is due to the  message update and message accumulation steps, the last term is from the inverse of a matrix for the large region.
For update rules (\ref{Two_layers_GaBP_rules}), we obtain
\begin{equation}\label{computational_gain}
    \#_2 = \sum_{i=1}^{M}\left[ \frac{3}{2} n_i^3 + (M-1)n_i \left(n_i + 1\right)\left(p_i - 1\right)\right] + M\frac{3}{2}N^3.
\end{equation}
Since
\begin{equation}
    \#_2 - \#_1 = (M-3)\sum_{i=1}^{M}n_i \left(n_i + 1\right)\left(p_i - 1\right) + (M-1)\frac{3}{2}N^3,
\end{equation}
one obtains a speed-up if $M>3$ for an arbitrary region graph. For certain regular partitions, for example the one in \cref{fig:two_layers_network:a}, $M$ scales like $N$, and in this cases, the \cref{algorithm:non_symmetric_two_layers_GaBP} performs $O(N^3)$ operations whereas rules (\ref{Two_layers_GaBP_rules}) perform $O(N^4)$ operations for each large region.

It is easy to see that if the LU method is employed, the number of operations for the single sweep is $O(K)$, where $K$ is an overall number of variables, only if the number of variables, for some regular partition for which the limit makes sense, in each large region scales like $O(1)$. Clearly, LU is not the best option for all cases, for example, matrices can have a particular structure (tridiagonal as \cref{fig:two_layers_network:a}), or it may be more advantageous to use probing or other techniques of estimation of certain subblocks of the inverse matrix \cite{tang2012probing} in combination with some iterative scheme for the solution of linear system.

The \cref{algorithm:non_symmetric_two_layers_GaBP} can be  justified theoretically on the basis of the following two theorems.

\begin{restatable}{theorem}{generalizedGaBPconsistency}
\label{prop:generalized_GaBP_Consistency}
If there is $N\in\mathbb{N}$ such that $\mathbf{m}^{(N+k)}_{e} = \mathbf{m}^{(N)}_{e}$, $\mathbf{\Lambda}^{(N+k)}_{e} = \mathbf{\Lambda}^{(N)}_{e}$ for all $e\in\mathcal{E}_{\mathcal{R}}$ and for any $k\in\mathbb{N}$, then for each large region $\left[\mathbf{x}\right]_{L} \equiv \widetilde{\mathbf{\Lambda}}_0^{-1}\widetilde{\mathbf{b}}_0 = \left[\mathbf{A}^{-1}\mathbf{b}\right]_{L}$ (see \cref{algorithm:non_symmetric_two_layers_GaBP} for details).
\end{restatable}
That is, the steady state of the message flow, if it exists, corresponds to the exact solution. The proof is given in \ref{Appendices:Consistency_of_generalized_GaBP}.

For the sufficient condition for convergence we need additional definitions. First, based on a given region graph $\left\{\mathcal{R}, \mathcal{E}_{\mathcal{R}}\right\}$, we define a set of variable subsets
\begin{equation}\label{partition}
    F \equiv \left(\underset{i}{\cup}\left\{l_i\right\}\right)\cup\left(\underset{j}{\cup}\left\{L_j\backslash \underset{p\in C(L_j)}{\cup}\left\{l_p\right\}\right\}\right),
\end{equation}
where $l_i$ is a small region and $L_j$ is a large region. Using $F$, we form a partition of the matrix $\mathbf{A}$ and the right-hand-side vector $\mathbf{b}$
\begin{equation}\label{matrix_partition}
    \mathbf{A} = 
    \begin{pmatrix}
        \mathbf{A}_{ii}&\mathbf{A}_{ij}&\hdots\\
        \mathbf{A}_{ji}&\mathbf{A}_{jj}&\hdots\\
        \vdots&\vdots&\ddots\\
    \end{pmatrix},
    \mathbf{b} = 
    \begin{pmatrix}
        \mathbf{b}_{i}\\
        \mathbf{b}_{j}\\
        \vdots\\
    \end{pmatrix},
\end{equation}
where each diagonal block corresponds to the element of the set $F$. We also define
\begin{equation}\label{generalized_matrix_R}
    \begin{split}
    &\widetilde{\mathbf{A}}_{ij} = \mathbf{A}_{ii}^{-1}\mathbf{A}_{ij}\equiv \mathbf{I}_{ij} - \widetilde{\mathbf{R}}_{ij},\\
    &\left\|\widetilde{\mathbf{R}}\right\|_{ij} \equiv \left\|\widetilde{\mathbf{R}}_{ij}\right\|,~\widetilde{\mathbf{b}}_{i} = \mathbf{A}_{ii}^{-1}\mathbf{b}_i.
    \end{split}
\end{equation}
Note that the second line in the preceding equation contains a definition of matrix $\left\|\widetilde{\mathbf{R}}\right\|$, which depends on the operator norm $\left\|\cdot\right\|$ (see \cite[ch. 5, Definition 5.6.3]{Horn:2012:MA:2422911}). 

The following statement gives sufficient conditions for convergence.
\begin{restatable}{theorem}{generalizedGaBPconvergence}
\label{prop:Generalized_GaBP_convergence}
    If for matrix (\ref{matrix_partition}) which is based on  partition (\ref{partition}) $\det\mathbf{A}_{ii}\neq 0~\forall i$ and $\rho\left(\left\|\widetilde{\mathbf{R}}\right\|\right)<1$ in some operator norm, then two-layer generalized GaBP (algorithm \ref{algorithm:non_symmetric_two_layers_GaBP}) converges to the exact solution $\mathbf{x} = \mathbf{A}^{-1}\mathbf{b}$.
\end{restatable}

The proof of this theorems appears in \ref{Appendices:Convergence_of_generalized_GaBP}. From the second part of the argument in \ref{Appendices:Convergence_of_generalized_GaBP:2}, one can deduce the following
\begin{corollary}
    Generalized GaBP (algorithm \ref{algorithm:non_symmetric_two_layers_GaBP}) converges whenever GaBP converges (\cref{prop:GaBP_convergence}) and all submatrices corresponding to the large blocks are invertible (see equations (\ref{partition}), (\ref{matrix_partition})).
\end{corollary}
The opposite does not hold. For example, consider a matrix
\begin{equation}\label{generalized_GaBP_example}
    \begin{split}
        &\mathbf{A} = 
        \begin{pmatrix}
            10 & 1.5 & 2 & 2 & 0 & 2 & 0\\
            2 & 4 & 2.5 & 0 & 2 & 0 & 0\\
            2 & 3 & 5 & 0 & 0 & 0 & 1\\
            2 & 0 & 0 & 10 & 0.5 & 1 & 0\\
            0 & 2 & 0 & 0.5 & 5 & 0 & 1\\
            2 & 0 & 0 & 1 & 0 & 7 & 1\\
            0 & 0 & 1 & 0 & 1 & 1 & 2
        \end{pmatrix}=
        \begin{pmatrix}
            \mathbf{A}_{11}&\mathbf{A}_{12}&\mathbf{A}_{13}\\
            \mathbf{A}_{21}&\mathbf{A}_{22}&\mathbf{A}_{23}\\
            \mathbf{A}_{31}&\mathbf{A}_{32}&\mathbf{A}_{33}
        \end{pmatrix},\\
        &\mathbf{A}_{11}\in \mathbb{R}^{3\times 3},~\mathbf{A}_{22}\in \mathbb{R}^{2\times 2},~\mathbf{A}_{33}\in \mathbb{R}^{2\times 2}.
    \end{split}
\end{equation}
In this case, the spectral radius of matrix $\left|\mathbf{\widetilde{R}}\right|$ defined in \cref{prop:GaBP_convergence} equals $\sim 1.03$ and GaBP diverges\footnote{Note that the divergence of GaBP does not follow from $\left|\mathbf{\widetilde{R}}\right|>1$  as \cref{prop:GaBP_convergence} provides only sufficient conditions. For this particular case, the pathological behavior of GaBP follows from the numerical experiment (see \cite{git_GaBP} for details).}. On the other hand, the spectral radius of $\left|\mathbf{\widetilde{R}}\right|$ defined by  (\ref{generalized_matrix_R}) and the partition given in  (\ref{generalized_GaBP_example}) is smaller than one in $l_{\infty}$ and spectral norms \cite[Examples 5.6.5, 5.6.6]{Horn:2012:MA:2422911} (see \cite{git_GaBP} for further details).

\section{GaBP as a smoother for the multigrid method}\label{Section:GaBP_as_a_smoother_for_the_multigrid_scheme}
The most straightforward view on the geometric multigrid is to describe it as an acceleration scheme for classical iterative methods. For completeness, we briefly recall the main ideas. 

The multigrid consists of four essential elements: a projection operator $\mathbf{I}_{V}^{V^{'}}:V\rightarrow V^{'}$ ($V$, $V^{'}$ are linear spaces) that reduces the number of degrees of freedom, an interpolation operator $\mathbf{I}_{V^{'}}^{V}:V^{'}\rightarrow V$ that acts in the "inverse" way, a smoothing operator $\mathbf{S}_{V}:V\rightarrow V$ which is usually a classical relaxation method, and a set of linear operators $\mathbf{A}_{V^{'}}$ that approximate $\mathbf{A}$ on coarse spaces $V^{'}$. What we describe next is a two-grid cycle.
\begin{itemize}
    \item For the current approximation $\mathbf{x}^{n}$ of solutions of $\mathbf{A}\mathbf{x} = \mathbf{b}$, one performs several relaxation steps $\overline{\mathbf{x}} = \mathbf{S}^{\nu} \mathbf{x}^{n}$.
    \item Then, based on properties of $\mathbf{S}$, the linear space $V^{'}$ and the projection operator $\mathbf{I}_{V}^{V^{'}}:V\rightarrow V^{'}$ are constructed. The purpose of this space is to represent the residual $\mathbf{r} = \mathbf{b} - \mathbf{A}\overline{\mathbf{x}}$ and an error $\mathbf{e} = \mathbf{x}_{\text{exact}} - \overline{\mathbf{x}}$ accurately using fewer degrees of freedom $\left|V^{'}\right|<|V|$.
    \item Having the space $V^{'}$, one constructs an operator $\mathbf{A}^{'}$ that approximates $\mathbf{A}$ and solves the  error equation $\mathbf{A}^{'}\mathbf{e}^{'}=\mathbf{I}_{V}^{V^{'}}\mathbf{r}$.
    \item The error, after projection back to $V$, gives the next approximation to the exact solution, $\mathbf{x}^{n+1} = \mathbf{S}^{\mu}\left(\mathbf{x}^{n} + \mathbf{I}^{V}_{V^{'}}\mathbf{e}^{'}\right)$.
\end{itemize}
The multigrid utilizes a two-grid cycle to solve the error equation $\mathbf{A}^{'}\mathbf{e}^{'}=\mathbf{I}_{V}^{V^{'}}\mathbf{r}$ itself. It produces the chain of spaces (grids in the geometric setup), projection operators that allow moving between them, and a set of approximate linear operators. For more details, we refer the reader to other resources: a simple introduction to geometric multigrid can be found in \cite[ch. 13]{saad2001iterative}, for the algebraic multigrid a recent review \cite{xu2017algebraic},  physical considerations about algebraic multigrid can be found in the introduction  of \cite{ron2011relaxation}, and among other books on the subject, \cite{trottenberg2000multigrid} provides a comprehensive introduction for practitioners.

\begin{algorithm}[t]
    \caption{GaBP as a smoother.}
    \label{algorithm:smoother_GaBP}
    \begin{algorithmic}
        \STATE Compute a residual $\mathbf{r}^{n} = \mathbf{b} - \mathbf{A}\mathbf{x}^{n}$.
        \STATE Apply $\mu$ sweeps of \cref{algorithm:non_symmetric_GaBP} or \ref{algorithm:non_symmetric_two_layers_GaBP} to the linear system $\mathbf{A}\mathbf{e}= \mathbf{r}^{n}$.
        \STATE Perform an error correction $\mathbf{x}^{n+1} = \mathbf{x}^{n} + \mathbf{e}^{\mu}$.
    \end{algorithmic}
\end{algorithm}

Here, we consider only linear systems of equations arising from finite difference discretization of elliptic equations with smooth coefficients in two space dimensions. In this case it is possible to use grids in place of linear spaces. Let the finest grid contain $2^{J} + 1$ points $\mathcal{G}_{J} = \left\{(i+j)h| h= 2^{-J}, i, j = \overline{0, 2^{J}}\right\}$, then the coarser grid $\mathcal{G}_{J-1}$ contains each second points along both directions. As we are working in the physical space, the restriction operator $\mathbf{I}_{V}^{V^{'}}$ computes a weighted average of neighbouring points, operator $\mathbf{I}^{V}_{V^{'}}$ performs interpolation, and $\mathbf{A}$ on the grid $\mathcal{G}_{J^{'}}$ is a finite difference approximation of the differential operator. In this article we always use full weighting restriction \cite[eq. 2.3.3]{trottenberg2000multigrid} and bilinear interpolation \cite[eq. 2.3.7]{trottenberg2000multigrid}.

The smoother should be a mapping $\mathbf{S}:\mathbf{x}^{n}\rightarrow \mathbf{x}^{n+1}$. Although GaBP is not of this form, one can use an error correction scheme as explained in \cref{algorithm:smoother_GaBP}.
In the next two subsections we analyze smoothing properties of \cref{algorithm:smoother_GaBP} and estimate its computational complexity.

\subsection{Local Fourier Analysis}
\label{Section:GaBP_as_a_smoother_for_the_multigrid_scheme.Subsection:Local_Fourier_Analysis}
Local Fourier Analysis allows us to compute the spectral radius of the  two-grid cycle, the smoothing factor of the relaxation scheme, and the  error contraction in a chosen norm \cite[ch. 4]{trottenberg2000multigrid}. In this subsection, we apply the analysis to the central difference discretization of the Laplace equation in two spatial dimensions
\begin{equation}
    \frac{1}{h^2}\left[\begin{matrix}
    &-1&\\
    -1&4&-1\\
    &-1&\\
    \end{matrix}\right]u_{ij} = f_{ij}.
\end{equation}
If after  $\mu$ sweeps of \cref{algorithm:non_symmetric_GaBP} the solution has a form $\mathbf{S} \mathbf{r}^{n}$, then the output of \cref{algorithm:smoother_GaBP} is $\mathbf{x}^{n+1} = \mathbf{S}\left(\mathbf{b} - \mathbf{A}\mathbf{x}^{n}\right) + \mathbf{x}^{n}$. Thus, for the error we have
\begin{equation}
    \mathbf{e}^{n+1} = \left(\mathbf{I} - \mathbf{S}\mathbf{A}\right)\mathbf{e}^{n}.
\end{equation}
Now we need to find a stencil of the operator $\mathbf{S}$. For GaBP it differs for parallel and sequential versions. For one and two sweeps of parallel version on the infinite lattice we have
\begin{equation}
    \mathbf{S}^{1}_{\text{parallel}} u_{ij} = \frac{h^2}{4} u_{ij},
    \mathbf{S}^{2}_{\text{parallel}} u_{ij} =
    \frac{h^2}{12}\left[\begin{matrix}
        &1&\\
        1&4&1\\
        &1&\\
    \end{matrix}\right]u_{ij}.
\end{equation}
It is easy to compute symbols
\begin{equation}
    \begin{split}
        &\mathbf{S}^{1}_{\text{parallel}} (\boldsymbol{\theta}) = \frac{h^2}{4},\mathbf{S}^{2}_{\text{parallel}} (\boldsymbol{\theta}) = \frac{h^2}{3}\left(1 + \cos\left(\frac{\theta_1+\theta_2}{2}\right)\cos\left(\frac{\theta_1-\theta_2}{2}\right)\right),\\
        &\mathbf{A}(\boldsymbol{\theta}) = \frac{4}{h^2}\left(1 - \cos\left(\frac{\theta_1+\theta_2}{2}\right)\cos\left(\frac{\theta_1-\theta_2}{2}\right)\right).
    \end{split}
\end{equation}
As the smoothing factor is
\begin{equation}
    \mu = \max_{\theta\in \text{high}}\left|1 - \mathbf{S}\left(\boldsymbol{\theta}\right)\mathbf{A}\left(\boldsymbol{\theta}\right)\right|,~\text{high} = \left[-\pi,\pi\right[^2\backslash\left[-\frac{\pi}{2},\frac{\pi}{2}\right[^2,
\end{equation}
we can conclude that the parallel version of GaBP shows no smoothing properties. This conclusion is confirmed by our numerical experiments.

\begin{figure}[t]
    \centering
    \subfloat[]{\label{fig:convergence_anisotropic_problem:a}\includegraphics[scale=0.42]{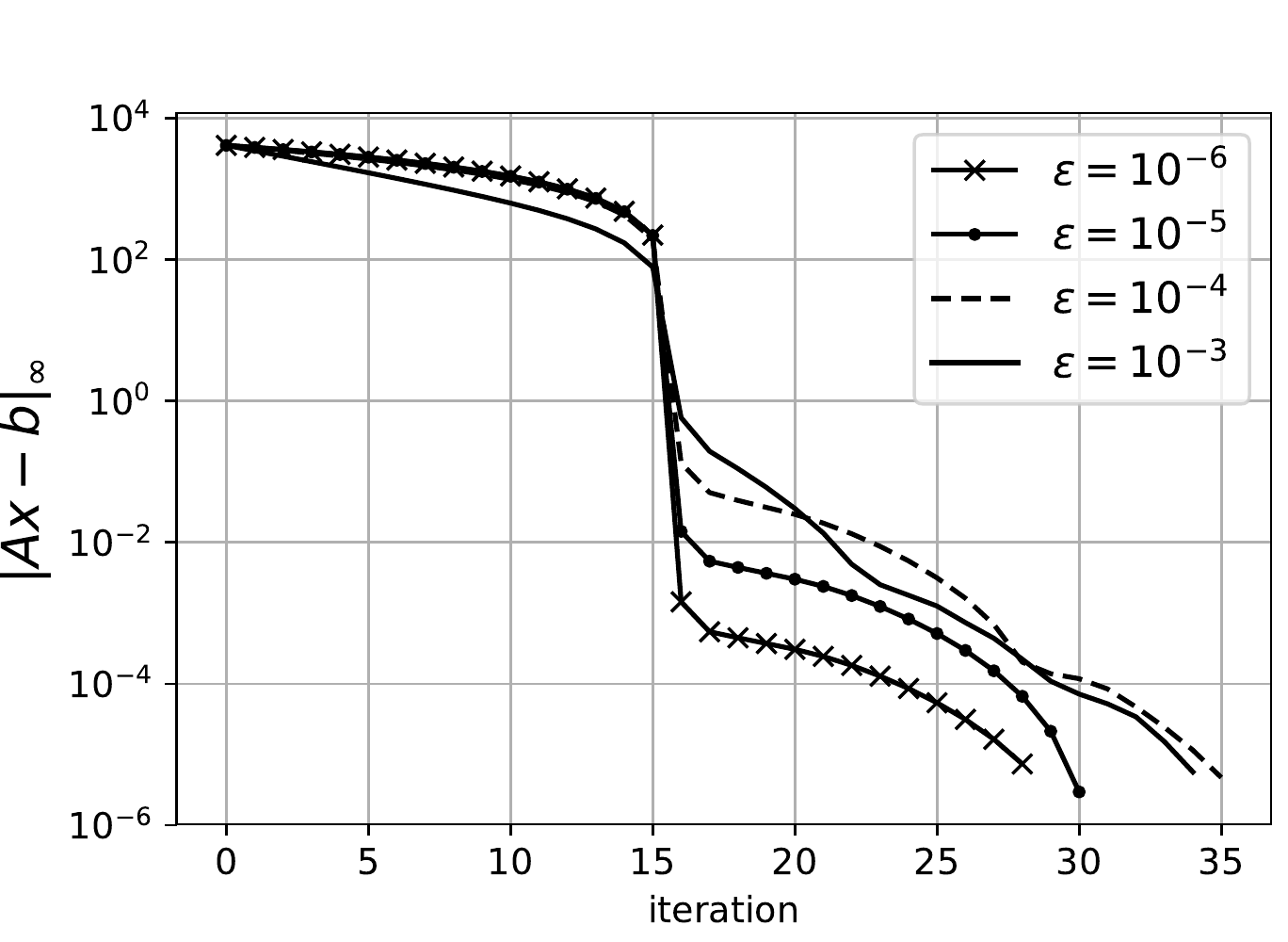}}\quad
    \subfloat[]{\label{fig:convergence_anisotropic_problem:b}\includegraphics[scale=0.42]{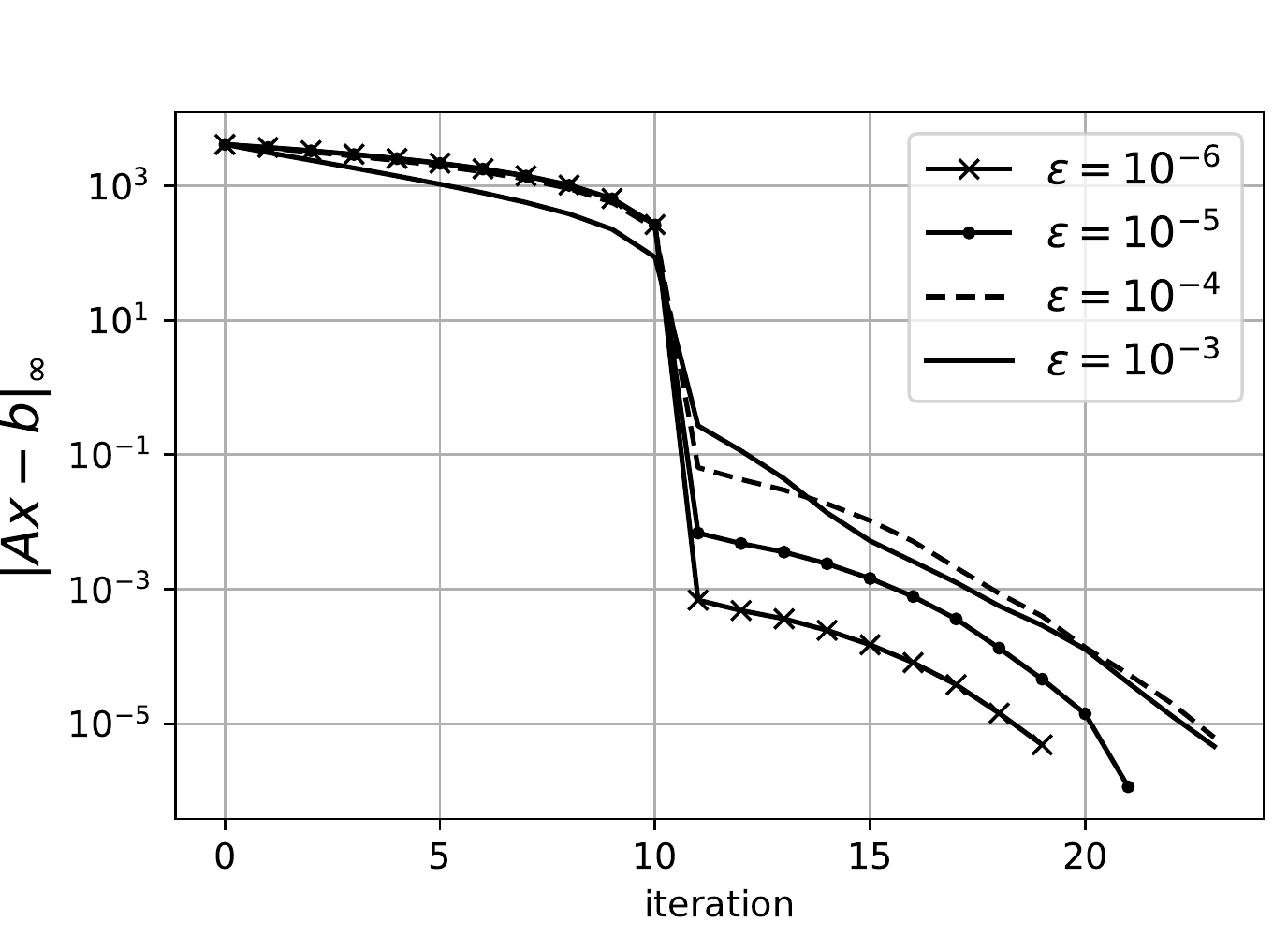}}
    \caption{Convergence histories for different anisotropies $\epsilon$: (a) two GaBP sweeps and (b) three GaBP sweeps for presmoothing and postsmoothing. In both cases, the fine grid consists of $2^{6}+1$ points, and the coarsest grid consists of $2^{3}+1$ points along each coordinate line. For comparison, if $\epsilon = 10^{-3}$, multigrid with Gauss-Seidel smoother (3 presmoothing and postsmoothing sweeps) converges after $\sim 400$ iterations. The sharp drops of the residual are the result of the cumulative effect that eludes explanation via Local Fourier Analysis. Namely, for small epsilon, vertical lines are effectively decoupled from each other. Scheme (a) needs $15$ iterations to solve exactly systems of linear equations for each line, and scheme (b) need $10$ iterations. See \cref{Section:GaBP_as_a_smoother_for_the_multigrid_scheme.Subsection:Local_Fourier_Analysis} for details.}
    \label{fig:convergence_anisotropic_problem}
\end{figure}

In the sequential case, one can deduce the form of $\mathbf{S}$ based on the elimination perspective. When one starts to move along the lattice, messages correspond to the elimination of variables, which means that 
\begin{equation}
    \mathbf{S}^{-1}_{\text{sequential}} = 
    \frac{1}{h^2}\left[\begin{matrix}
        &&\\
        -1&4&~\\
        &-1&\\
    \end{matrix}\right].
\end{equation}
Then, for the smoothing factor we have
\begin{equation}
    \mu = \max_{\theta\in \text{high}}\sqrt{\left|\frac{\cos{\left (\theta_{1} - \theta_{2} \right )} + 1}{4 \cos{\left (\theta_{1} \right )} + 4 \cos{\left (\theta_{2} \right )} - \cos{\left (\theta_{1} - \theta_{2} \right )} - 9}\right|} = \frac{1}{2},
\end{equation}
for $\theta_2 = \frac{\pi}{2}$ and $\theta_1 = 2\arctan\frac{1}{3}$. This means that the smoothing factor for the sequential GaBP coincides with the one for sequential Gauss-Seidel iteration scheme \cite[Example 4.3.4]{trottenberg2000multigrid}. It is also clear that for the anisotropic problem
\begin{equation}
    \frac{1}{h^2}\left[\begin{matrix}
    &-1&\\
    -\epsilon&2(1+\epsilon)&-\epsilon\\
    &-1&\\
    \end{matrix}\right]u_{ij} = f_{ij},
\end{equation}
both Gauss-Seidel scheme and sequential GaBP lose their smoothing properties when $\epsilon\rightarrow 0$. However, numerical experiments  (figure~\ref{fig:convergence_anisotropic_problem}) show that the convergence rate of GaBP does not depend on $\epsilon$. An explanation for this particular case is straightforward. For sufficiently small $\epsilon$ equations for each vertical line (i.e., in $y$ direction) are independent. GaBP is an exact solver for trees. The single multigrid iteration eliminates variables from $2+2=4$ neighbours in case \ref{fig:convergence_anisotropic_problem:a} and from $3+3=6$ neighbours in case \ref{fig:convergence_anisotropic_problem:b}. When messages cover the whole line of $2^{6}-1 = 63$ nodes, the system of linear equations for each vertical line is solved exactly. It gives  $63/4\sim 15$ iterations for \ref{fig:convergence_anisotropic_problem:a}, and $63/6\sim 10$ iterations for \ref{fig:convergence_anisotropic_problem:b}.

\begin{table}
\centering
    \begin{tabular}{c|c|c|c|c|c}
    Sweeps & Stencil & GaBP & line GaBP & GS & $x/y$-GS\\\hline
    \multirow{2}{1em}{$1$} & 5 points & $18N$& $38N$ & $9N$ & $14N$ \\\cline{2-6}
    & 9 points & $32N$& N/A & $17N$ & $21N$ \\\hline
    \multirow{2}{1em}{$2$} & 5 points & $30N$ & $65N$ & $18N$ & $28N$  \\\cline{2-6}
    & 9 points & $56N$& N/A & $34N$  & $42N$\\\hline
    \multirow{2}{3em}{$M\geq3$} & 5 points & $12N\cdot M +6N$& $28N\cdot M + 9N$ & $9N\cdot M$ & $14N\cdot M$ \\\cline{2-6}
    & 9 points & $24N\cdot M + 8N$& N/A & $17N\cdot M$ & $21N\cdot M$ \\\hline
    \end{tabular}
    \caption{Computational complexity of GaBP (error correction scheme) with precomputed $\Lambda$ messages in comparison with the classical Gauss-Seidel relaxation schemes. Line GaBP refers to the partition presented in \cref{fig:two_layers_network:a} and $x/y$-GS is a classical line smoother. As one can see from the theory of generalized GaBP, it is not possible to apply line GaBP for $9$ points stencil, because large regions do not cover all edges of the original graph. However, one indeed can construct line GaBP smoothers for this case, too, but we do not consider them here.}
    \label{fig:computational_complexity}
\end{table}

The effect displayed in \cref{fig:convergence_anisotropic_problem} is a manifestation of the dynamic nature of GaBP. Even as part of the multigrid it maintains information about all previous iterations. More convergence histories can be found below, in the section with numerical examples. Overall, we  conclude that sequential GaBP as part of the  multigrid behaves similarly to Gauss-Seidel in the absence of anisotropy, but is substantially more robust in the presence of anisotropy. The behavior captured in \cref{fig:convergence_anisotropic_problem} also illustrates that Local Fourier Analysis is not an appropriate tool to analyze GaBP.

\subsection{Reducing computational complexity}
\label{Section:GaBP_as_a_smoother_for_the_multigrid_scheme.Subsection:Reducing_the_computational_complexity}
The number of floating point operations per iteration for algorithms~\ref{algorithm:non_symmetric_GaBP} and \ref{algorithm:non_symmetric_two_layers_GaBP} depends on the graph of the matrix $A$. Here, we consider the operator with the dense $9$ point stencil
\begin{equation}
    A = \left[\begin{matrix}
    *&*&*\\
    *&*&*\\
    *&*&*\\
    \end{matrix}\right],
\end{equation}
which can come from the second order finite difference approximation of a  differential operator containing second and first derivatives. The same analysis for the 5 points stencil is straightforward. For convenience, we split \cref{algorithm:non_symmetric_GaBP} (sequential version) into three parts:
\begin{itemize}
    \item Accumulation stage.  $\Sigma$ and $m$ are computed.
    \item Update stage. New messages $\widetilde{\Lambda}$ and $\widetilde{\mu}$ are constructed from the previous ones.
    \item Termination stage. The final answer $m/\Sigma$ is obtained.
\end{itemize}
We also neglect all effects from boundaries. Under these assumptions, the  number of floating point operations for the single sweep GaBP is
\begin{equation}
    \#_{\text{GaBP}_1}=\underbrace{4N+8N}_{\text{accumulate}} + \underbrace{8N+12N}_{\text{update}} + \underbrace{N}_{\text{terminate}} + \underbrace{18N}_{\mathbf{r} = \mathbf{b} - \mathbf{A}\mathbf{x}_0} + \underbrace{N}_{+\mathbf{x}_0} = 52N.
\end{equation}
Here, we perform only a half of accumulation stage and a half of update stage, because we do not need to receive messages from nodes that we have not visited yet, nor we need to send messages to already visited nodes.

For lexicographical Gauss-Seidel scheme, the number of floating point operations is $\#_{\text{LEX GS}} = 17N.$ It means that a single sweep of \ref{algorithm:non_symmetric_GaBP} takes slightly fewer floating point operations than three sweeps of Gauss-Seidel smoother $\#_{\text{GaBP}_1}\sim 3\#_{\text{LEX GS}}$. For $M$ sweeps of GaBP one has $\#_{\text{GaBP}_M}= N(64M-12)$. In the context of multigrid, it is important to have a cheap smoother, but even $52N$ is too expensive. However, it is possible to reduce computational complexity by precomputing all required messages $\widetilde{\Lambda}$, which depend only on the matrix $\mathbf{A}$ and not on the right-hand-side vector. For $M$ sweeps of GaBP with precomputed $\widetilde{\Lambda}$, we have $\#^{\widetilde{\Lambda}}_{\text{GaBP}_M}= N(24M+8)$. We summarize  all these results regarding the complexity of GaBP in  Table~\ref{fig:computational_complexity}.
\section{Numerical examples}\label{Section:Numerical_examples}
In this section, we present numerical experiments with matrices that arise from second-order finite difference approximations of two-dimensional elliptic differential equations with $(x, y) \in \left[0, 1\right]^2$. The grid is assumed to be uniform and consists of $2^6 -1 $ inner points along each direction. We use Dirichlet boundary conditions in all the examples. These conditions are not specified directly and should be extracted from the exact solution. In the same vein, the form of the source term $g(x, y)$ (ride-hand side) can be derived from the exact solution and is not given explicitly. In all the experiments, the stopping criterion is $\left\|r\right\|_{\infty}\leq2\cdot 10^{-4}$. In the tables below, $N_{\text{it}}$ denotes the number of iterations and $N$ is the number of variables. Before we begin the main discussion, we summarize the main  properties of the solvers that are used.
\subsection{Note about solvers and smoothers}
In addition to the number of floating point operations (FLOP) (see \cref{fig:computational_complexity}), an important characteristic of a solver is its degree of parallelism, which is provided for various solvers  in the following table (for classical methods see e.g. \cite[ch. 6]{trottenberg2000multigrid}): \\ 
\begin{center}
    \begin{tabular}{c|c}
    Solvers &  Process in parallel\\\hline
    parallel GaBP, Jacobi & all point \\\hline
    sequential GaBP, GS & the single point \\\hline
    red-black GaBP, red-black GS & the half of all points \\\hline
    4-colors GaBP, 4-colors GS & the quarter of all points\\\hline
    $x-$ or $y-$GS & a single line\\\hline
    zebra-line GS, alternating zebra GS & a half of all lines\\\hline
    line GaBP & all lines \\\hline
    \end{tabular}
\end{center}
\vspace{0.5cm} 
Note that the line version of GaBP possesses a better degree of parallelism than GS versions. When we consider GaBP as a multigrid component, we always use V-cycle, bilinear restriction, and prolongation operators and LU as a coarse-grid solver. In our notation $V(J_1, J_2)$ means that the fine grid consists of $2^{J_1} - 1$ points, the coarse grid of $2^{J_1 - (J_2 - 1)} - 1$ points; numbers $(n, m)$ before the smoother name refer to the number of pre- and post-smoothing steps.
\subsection{GaBP as a stand-alone solver}
Classical relaxation methods are rarely used outside the AMG (algebraic multigrid) or GMG (geometric multigrid) to solve linear systems. Nevertheless, we present an example of their performance below. As a linear problem we use the following elliptic boundary value problem:
\begin{equation}\label{stand_alone}
    \begin{split}
        &\left(a(x, y)\frac{\partial^2}{\partial x^2} + b(x, y)\frac{\partial^2}{\partial y^2} + \alpha(x, y)\frac{\partial}{\partial x} + \beta(x, y)\frac{\partial}{\partial y}\right)\phi = g,\, (x, y)\in\left[0, 1\right]^2;\\
        &a = e^{-x(y+2)} + 10, \, \alpha = \cos\left(\pi\left(x + \frac{y}{2}\right)\right)\cos(2\pi x) + 4;\\
        &b = e^{-2x+2y}\cos^2\left(2\pi\left(2x+\frac{y}{2}\right)\right) + 3, \, \beta = e^{2x - 2y}
    \end{split}
\end{equation}
with $g(x,y)$ and boundary conditions chosen such that $\phi_{\text{exact}} = \cos(\pi x)\cos(\pi y)$ is the exact solution (this method of manufactured solutions is used in the remaining examples as well). The performance of various methods on this problem is shown in the following table:\\
\begin{center}
    \begin{tabular}{c|c|c}
    Solver &  $N_{\text{it}}$& FLOP, $10^3\cdot N$\\\hline
    sequential GaBP &$1548$ & $99$ \\\hline
    parallel GaBP & $3299$ & $211$ \\\hline
    GS & $3102$ & $53$\\\hline
    4-colors GS & $2620$ & $45$ \\\hline
    4-colors GaBP & $1865$ & $119$ \\\hline
    Jacobi & $4746$ & $81$ \\\hline
    error correction 4-colors GaBP $(3)$ & $706$ & $56$ \\\hline
    \end{tabular}
\end{center}
\vspace{0.5cm} 
The last line in the table corresponds to the three sweeps of the error correction scheme with precomputed messages $\Lambda$ (see \cref{Section:GaBP_as_a_smoother_for_the_multigrid_scheme.Subsection:Reducing_the_computational_complexity}). We see that GaBP does not provide particular advantages over classical relaxation methods as a stand-alone solver, even though there are some reports of its  excellent performance (see, e.g., \cite{weiss2000correctness}, \cite{bickson2008gaussian}). The main bottleneck here is the computational complexity of the scheme. To some extent, one can mitigate this problem by precomputing $\Lambda$ before the iteration process begins. Still, in this particular situation, both the 4-colors GS and Jacobi provide better alternatives due to their low cost and high degree of parallelism. We observe the analogous behavior for other elliptic problems.
\subsection{GaBP as a multigrid smoother}
As a rule, relaxation solvers become applicable to real large-scale problems and are competitive with projection methods only in the framework of multigrid schemes. We now present several situations that could potentially challenge state-of-the-art GMG smoothers. Note that in this section, we always use the error correction version of GaBP. Since we apply the same version of multigrid, we compute the FLOP score solely for the smoother and on the fine level only.
\subsection{Large mixed derivative}
The first equation of interest is of the form
\begin{equation}\label{mixed_term}
    \begin{split}
        &\left(\frac{\partial^2}{\partial x^2} + \frac{\partial^2}{\partial y^2} + \left(2-\epsilon\right)\frac{\partial^2}{\partial x\partial y}\right)\phi = g,~(x, y)\in\left[0, 1\right]^2,\\
        &\phi_{\text{exact}} = 2x^3 y^4.
    \end{split}
\end{equation}
The main problem here is that for small $\epsilon$, the ellipticity is almost lost. Below one can see the table with the best in terms of FLOP $V$-cycle solver of each kind:\\
\begin{center}
    \begin{tabular}{c|c|c||c|c}
        &\multicolumn{2}{|c||}{$\epsilon = 0.01$}& \multicolumn{2}{|c}{$\epsilon = -0.01$}\\\hline
        Solver, $V(6, 6)$ &  $N_{\text{it}}$& FLOP, $N$ & $N_{\text{it}}$& FLOP, $N$\\\hline
        4-color GaBP $(0, 4)$ &$23$ & $2392$ & $28$ & $2917$ \\\hline
        4-color GS  $(1, 1)$ &$70$ & $2380$ & $84$ & $2856$  \\\hline
        zebra-line GS $(0, 1)$ & $104$ & $2184$ & $127$ & $2667$  \\\hline
        alternating-zebra GS $(0, 1)$ & $64$ & $2688$ & $78$ & $3276$
    \end{tabular}
\end{center}
\vspace{0.5cm} 
We can see that the performance of GaBP is comparable with the 4-color GS smoother. So GaBP can be considered to be robust for the almost non-elliptic equations. We also stress that both 4-color GS and GaBP are preferable over the line smoothers for this problem because of their better degree of parallelism.

\subsection{Boundary layers}
The other practically relevant case that is a challenge for standard geometrical smoothers is an advection-diffusion problem in which advection dominates. We take as an example the following problem:
\begin{equation}\label{large_advection}
    \begin{split}
        &\left(-\epsilon\frac{\partial^2}{\partial x^2} - \epsilon\frac{\partial^2}{\partial y^2} + \frac{\partial}{\partial x} + \frac{\partial }{\partial y}\right)\phi(x, y) = 0,\,~(x, y)\in\left[0, 1\right]^2,\\
        &\phi_{\text{exact}} = \frac{2 e^{-1\big/\epsilon} - e^{(x-1)\big/\epsilon} - e^{(y-1)\big/\epsilon}}{e^{-1\big/\epsilon} - 1}.
    \end{split}
\end{equation}
As one can see, there are two boundary layers near $x=1$ and $y=1$, each of  width~$\sim \epsilon$. The solution is not large, but the derivative is $\sim 1\big/\epsilon$. The table below shows the performance results.\\
\begin{center}
    \begin{tabular}{c|c|c||c|c}
    &\multicolumn{2}{|c||}{$\epsilon = 0.02$}& \multicolumn{2}{|c}{$\epsilon = 0.01$}\\\hline
    Solver, $V(6, 6)$ &  $N_{\text{it}}$& FLOP, $N$ & $N_{\text{it}}$& FLOP, $N$\\\hline
    red-black GaBP $(5,0)$ & $5$ & $330$  & $3$  & $198$  \\\hline
    red-black GS  $\forall(n,m)$ & diverge &  & diverge &   \\\hline
    zebra-line GS $(2,2)$ & $5$ & $280$ & diverge &  \\\hline
    alternating-zebra GS $(1,1)$ & $4$ & $224$ & $3$ & $168$
    \\\hline
    line GaBP $(0,2)$ & $5$ & $325$ & $5$ & $325$
    \end{tabular}
\end{center}
\vspace{0.5cm} 
We observe two interesting trends. First, for $\epsilon = 0.02$, it is not possible to apply the red-black GS. However, it is still possible to construct a smoother from GaBP, using a sufficiently large number of sweeps. Second, the line GaBP smoother with a reasonable amount of steps performs nearly as well as the  alternating-zebra GS. Still, since the GABP can process all lines simultaneously, we conclude that it can outperform classical geometric smoothers for such advection-dominated elliptic problems. Furthermore, if the precomputation of $\Lambda$ is affordable from the perspective of the additional storage required, it is far better to use the red-black GaBP smoother.

\subsection{Inner layers}
We consider another advection-dominated diffusion problem
\begin{equation}\label{large_advection_2}
    \begin{split}
        &\left(\epsilon\frac{\partial^2}{\partial x^2} + \epsilon\frac{\partial^2}{\partial y^2} + x\frac{\partial}{\partial x} + y\frac{\partial }{\partial y}\right)\phi = g, \,(x, y)\in\left[0, 1\right]^2,\\
        &\phi_{\text{exact}} = e^{-(x+y-1)^2\big/\epsilon}.
    \end{split}
\end{equation}
This equation differs from (\ref{large_advection}) in two respects: 1) the solution has two inner layers, and 2) they are not aligned with the coordinate lines. Now the performance is as follows.\\
\begin{center}
    \begin{tabular}{c|c|c||c|c}
    &\multicolumn{2}{|c||}{$\epsilon = 0.015$}& \multicolumn{2}{|c}{$\epsilon = 0.01$}\\\hline
    Solver, $V(6, 6)$ &  $N_{\text{it}}$& FLOP, $N$ & $N_{\text{it}}$& FLOP, $N$\\\hline
    red-black GaBP $(3,0)$ & $7$ & $294$  & $13$  & $546$  \\\hline
    red-black GS  $\forall(n,m)$ & diverge &  & diverge &   \\\hline
    zebra-line GS $(2,0)$ & $9$ & $252$ & diverge &  \\\hline
    alternating-zebra GS $(1,1)$ & $4$ & $224$ & $5$ & $280$
    \\\hline
    line GaBP $(0,2)$ & $8$ & $520$ & $8$ & $520$
    \end{tabular}
\end{center}
\vspace{0.5cm} 
For this problem we can see the same pattern as for the previous example. The classical red-black solver is unable to smooth the error, whereas the GaBP-based color iteration scheme works fine. Additionally, due to its excellent degree of parallelism, the red-black GaBP significantly outperforms the  alternating-zebra GS smoother.

\subsection{Stretched grid}
Another situation of practical interest is given by the following problem:
\begin{equation}\label{stretched_grid}
    \begin{split}
        &\left(u\left(x|p, \eta\right)\frac{\partial^2}{\partial x^2} + u\left(y|p, \eta\right)\frac{\partial^2}{\partial y^2}\right)\phi = g, (x, y)\in\left[0, 1\right]^2,\\
        &u\left(x|p, \eta\right) =1 + \left(\left(x-\frac{1}{2}\right)^2 + \eta\right)^p\big/\epsilon,\\
        &\phi_{\text{exact}} = \cos(2\pi (x+y))\sin(2\pi (x-y)).
    \end{split}
\end{equation}
To understand this problem, consider the finite difference discretization of the Laplace equation on a grid which is highly concentrated near the edges of the domain. If one denotes $a_{\text{in}}$ and $a_{\text{bn}}$ to be characteristic scales of the coefficients, related to inner and boundary points respectively, the ratio $a_{\text{bn}}\big/a_{\text{in}}$ will be large. We achieve the same effect in equation~\eqref{stretched_grid} on the uniform grid by multiplying second derivatives by positive terms of the form $\left(1 + \left(\left(x-\frac{1}{2}\right)^2 + \eta\right)^p\big/\epsilon\right)$ which are approximately equal to $1$ inside the domain, but grow rapidly to large values near the boundaries. The table below shows thee results.\\
\begin{center}
    \begin{tabular}{c|c|c||c|c}
    &\multicolumn{2}{|c||}{$p, \eta, \epsilon = 20,1\big/2,10^{-6}$}& \multicolumn{2}{|c}{$p, \eta, \epsilon = 20,1\big/2,8\cdot 10^{-8}$}\\\hline
    Solver, $V(6, 6)$ &  $N_{\text{it}}$& FLOP, $N$ & $N_{\text{it}}$& FLOP, $N$\\\hline
    red-black GaBP $(3,0)$ & $18$ & $756$  & $23$  & $966$  \\\hline
    red-black GS  $(3,0)$ & $68$ & $1836$  & $97$  & $2619$  \\\hline
    zebra-line GS $(4,0)$ & $50$ & $2800$ & $72$ & $4032$  \\\hline
    alternating-zebra GS $(1,1)$ & $10$ & $560$ & $12$ & $672$
    \\\hline
    line GaBP $(0,2)$ & $20$ & $1300$ & $23$ & $1495$
    \end{tabular}
\end{center}
\vspace{0.5cm}
In this table, the first set of parameters $p, \eta, \epsilon = 20,1\big/2,10^{-6}$ corresponds to the linear stretching of the grid by a factor of $\sim 40$ and the second $p, \eta, \epsilon = 20,1\big/2,8\cdot10^{-8}$ to the linear stretching by a factor of $\sim 160$. We conclude that there is a version of the red-black GaBP with excellent convergence rate and good degree of parallelism. The same is true for the line GaBP smoother. Both of them can be used as an alternative to the classical alternating-zebra smoother.

\subsection{Comparison with a projection method}
For the sake of completeness, we also give an example of the performance of BiCGSTAB and GaBP-based multigrid for problem~\eqref{stand_alone}.\\
\begin{center}
    \begin{tabular}{c|c|c}
    Solver &  $N_{\text{it}}$& FLOP, $10^3\cdot N$\\\hline
    $V(6, 6)$, 4-color GaBP $(1, 1)$ & $21$ & $\sim 3$ \\\hline
    BiCGSTAB & $255$ & $\sim 38$ \\\hline
    \end{tabular}
\end{center}
\vspace{0.5cm} 
As one may have anticipated, the  projection method without a suitable preconditioner cannot outperform the  geometric multigrid.

Overall, based on the presented result, we conclude that different versions of GaBP perform either comparably well or better than the state-of-the-art smoothers for GMG. The main disadvantage of GaBP is its computational complexity. Even though with a precomputed $\Lambda$, one can substantially decrease the number of FLOPs, the cost of a single iteration is still higher than for the classical smoothers. However, its clear advantages are the robustness and the degree of parallelism. The former allows one to construct new point-based relaxation schemes for situations where classical point-based relaxation methods  fail. And the latter enables the line GaBP to outperform the alternating-zebra GS smoother.
\section{Conclusions}\label{Section:Conclusion}

In this paper, we have introduced a new class of solvers for linear systems that are  based on the generalized belief propagation algorithm. The solvers work for both symmetric and nonsymmetric  matrices. We show how to reduce the complexity of the resulting algorithm in comparison to that of the straightforward  application of the generalized belief propagation. A clear connection between the block LU decomposition and the new algorithm is established. Existing  proofs for symmetric systems are generalized to nonsymmetric systems, and two new proofs for a block version of the  GaBP are given. Furthermore, we show how to use the geometric multigrid to accelerate the GaBP, which with a precomputed $\widetilde{\Lambda}$  results in a robust solver with the same computational complexity as the one based on the Gauss-Seidel smoother.

We have demonstrated the performance of the new algorithm with several examples of boundary-value problems of varying complexity.  The  numerical experiments show that the GaBP is a competitive alternative to classical relaxation schemes. The reason for the good performance of GaBP is that it retains some information about all the previous stages, whereas the Gauss-Seidel, Jacobi, and Richardson solvers do not. Even though large computational overhead is a disadvantage of GaBP, the problem can be alleviated  within the framework of the multigrid scheme at the expense of additional storage and precomputation of some messages. Moreover, as part of the multigrid, the GaBP not only smooths  high-frequency components of the error, but also effectively decreases the low-frequencies. Our numerical experiments indicate that this feature promotes additional robustness. For example, the convergence rate of GaBP for anisotropic model elliptic problem given by $\left(\partial_x^2 + \epsilon\partial_y^2\right)u(x, y)=0$ does not depend on $\epsilon$, which is a somewhat unexpected result. Moreover, in the case of sharp inner and boundary layers, stretched grids, and large mixed derivatives, GaBP retains smoothing properties and performs better than the state-of-the-art geometrical smoothers.

The generalized GaBP introduced in the present work is in some sense a block version of the regular GaBP. Our numerical experiments show that the convergence rate increases with the size of the blocks such that after certain  scale the generalized GaBP can compete with Krylov subspace methods (see a numerical example at \cite{git_GaBP}). However, the computational cost increases as well. In practical applications, one should balance these two tendencies to construct an optimal solver. Some considerations about the choice of the blocks can be found in \cite{welling2004choice}, \cite{yedidia2001generalized}, but the issue is not yet fully resolved.

The generalized GaBP and GaBP as its particular case come from a domain of variational inference. To the best of our knowledge, there is currently no systematic analysis of the general relationship between deterministic/probabilistic inference and linear solvers. In our opinion, such a link may be useful for new interpretation of known techniques and provide insights that may lead to more efficient algorithms for numerical linear algebra.
\section*{Acknowledgement}
The author is indebted to Dr. Aslan Kasimov for valuable suggestions.

\bibliographystyle{plain}
\bibliography{refs.bib}

\begin{thebibliography}{10}

\bibitem{git_GaBP}
{https://github.com/VLSF/GaBP\_solvers}.

\bibitem{amann2005analysis}
Herbert Amann, Joachim Escher, Silvio Levy, and Matthew Cargo.
\newblock {\em Analysis}, volume~1.
\newblock Springer, 2005.

\bibitem{an1988note}
Guozhong An.
\newblock A note on the cluster variation method.
\newblock {\em Journal of Statistical Physics}, 52(3-4):727--734, 1988.

\bibitem{bartels2018probabilistic}
Simon Bartels, Jon Cockayne, Ilse~CF Ipsen, and Philipp Hennig.
\newblock Probabilistic linear solvers: A unifying view.
\newblock {\em arXiv preprint arXiv:1810.03398}, 2018.

\bibitem{bickson2008gaussian}
Danny Bickson.
\newblock Gaussian belief propagation: Theory and aplication.
\newblock {\em arXiv preprint arXiv:0811.2518}, 2008.

\bibitem{Bishop:2006:PRM:1162264}
Christopher~M. Bishop.
\newblock {\em Pattern Recognition and Machine Learning (Information Science
  and Statistics)}.
\newblock Springer-Verlag, Berlin, Heidelberg, 2006.

\bibitem{cockayne2018bayesian}
Jon Cockayne, Chris~J Oates, Ilse~CF Ipsen, Mark Girolami, et~al.
\newblock A bayesian conjugate gradient method.
\newblock {\em Bayesian Analysis}, 2018.

\bibitem{davis2016survey}
Timothy~A Davis, Sivasankaran Rajamanickam, and Wissam~M Sid-Lakhdar.
\newblock A survey of direct methods for sparse linear systems.
\newblock {\em Acta Numerica}, 25:383--566, 2016.

\bibitem{el2012relaxed}
Yousef El-Kurdi, Dennis Giannacopoulos, and Warren~J Gross.
\newblock Relaxed gaussian belief propagation.
\newblock In {\em 2012 IEEE International Symposium on Information Theory
  Proceedings}, pages 2002--2006. IEEE, 2012.

\bibitem{Elidan:2006:RBP:3020419.3020440}
Gal Elidan, Ian McGraw, and Daphne Koller.
\newblock Residual belief propagation: Informed scheduling for asynchronous
  message passing.
\newblock In {\em Proceedings of the Twenty-Second Conference on Uncertainty in
  Artificial Intelligence}, UAI'06, pages 165--173, Arlington, Virginia, United
  States, 2006. AUAI Press.

\bibitem{golub2012matrix}
Gene~H Golub and Charles~F Van~Loan.
\newblock {\em Matrix computations}, volume~3.
\newblock JHU press, 2012.

\bibitem{hennig2015probabilistic}
Philipp Hennig.
\newblock Probabilistic interpretation of linear solvers.
\newblock {\em SIAM Journal on Optimization}, 25(1):234--260, 2015.

\bibitem{heskes2004uniqueness}
Tom Heskes.
\newblock On the uniqueness of loopy belief propagation fixed points.
\newblock {\em Neural Computation}, 16(11):2379--2413, 2004.

\bibitem{Hestenes&Stiefel:1952}
M.~R. Hestenes and E.~Stiefel.
\newblock Methods of conjugate gradients for solving linear systems.
\newblock {\em Journal of research of the National Bureau of Standards},
  49:409--436, 1952.

\bibitem{Horn:2012:MA:2422911}
Roger~A. Horn and Charles~R. Johnson.
\newblock {\em Matrix Analysis}.
\newblock Cambridge University Press, New York, NY, USA, 2nd edition, 2012.

\bibitem{kikuchi1951theory}
Ryoichi Kikuchi.
\newblock A theory of cooperative phenomena.
\newblock {\em Physical review}, 81(6):988, 1951.

\bibitem{malioutov2006walk}
Dmitry~M Malioutov, Jason~K Johnson, and Alan~S Willsky.
\newblock Walk-sums and belief propagation in {G}aussian graphical models.
\newblock {\em Journal of Machine Learning Research}, 7(Oct):2031--2064, 2006.

\bibitem{minka2001expectation}
Thomas~P Minka.
\newblock Expectation propagation for approximate bayesian inference.
\newblock In {\em Proceedings of the Seventeenth conference on Uncertainty in
  artificial intelligence}, pages 362--369. Morgan Kaufmann Publishers Inc.,
  2001.

\bibitem{owhadi2017multigrid}
Houman Owhadi.
\newblock Multigrid with rough coefficients and multiresolution operator
  decomposition from hierarchical information games.
\newblock {\em SIAM Review}, 59(1):99--149, 2017.

\bibitem{pearl2014probabilistic}
Judea Pearl.
\newblock {\em Probabilistic reasoning in intelligent systems: networks of
  plausible inference}.
\newblock Elsevier, 2014.

\bibitem{pelizzola2005cluster}
Alessandro Pelizzola.
\newblock Cluster variation method in statistical physics and probabilistic
  graphical models.
\newblock {\em Journal of Physics A: Mathematical and General}, 38(33):R309,
  2005.

\bibitem{plarre2004extended}
Kurt~Hermann Plarre and PR~Kumar.
\newblock Extended message passing algorithm for inference in loopy gaussian
  graphical models.
\newblock {\em Ad Hoc Networks}, 2(2):153--169, 2004.

\bibitem{ron2011relaxation}
Dorit Ron, Ilya Safro, and Achi Brandt.
\newblock Relaxation-based coarsening and multiscale graph organization.
\newblock {\em Multiscale Modeling \& Simulation}, 9(1):407--423, 2011.

\bibitem{saad2003iterative}
Y.~Saad.
\newblock {\em Iterative Methods for Sparse Linear Systems}.
\newblock Society for Industrial and Applied Mathematics, Philadelphia, PA,
  USA, 2nd edition, 2003.

\bibitem{saad2001iterative}
Yousef Saad and Henk~A Van Der~Vorst.
\newblock Iterative solution of linear systems in the 20th century.
\newblock In {\em Numerical Analysis: Historical Developments in the 20th
  Century}, pages 175--207. Elsevier, 2001.

\bibitem{shental2008gaussian}
Ori Shental, Paul~H Siegel, Jack~K Wolf, Danny Bickson, and Danny Dolev.
\newblock Gaussian belief propagation solver for systems of linear equations.
\newblock In {\em Information Theory, 2008. ISIT 2008. IEEE International
  Symposium on}, pages 1863--1867. IEEE, 2008.

\bibitem{shewchuk1994introduction}
Jonathan~Richard Shewchuk et~al.
\newblock An introduction to the conjugate gradient method without the
  agonizing pain, 1994.

\bibitem{stewart1998matrix}
Gilbert~W Stewart.
\newblock {\em Matrix Algorithms: Volume 1: Basic Decompositions}, volume~1.
\newblock Siam, 1998.

\bibitem{sudderth2004embedded}
Erik~B Sudderth, Martin~J Wainwright, and Alan~S Willsky.
\newblock Embedded trees: Estimation of gaussian processes on graphs with
  cycles.
\newblock {\em IEEE Transactions on Signal Processing}, 52(11):3136--3150,
  2004.

\bibitem{tang2012probing}
Jok~M Tang and Yousef Saad.
\newblock A probing method for computing the diagonal of a matrix inverse.
\newblock {\em Numerical Linear Algebra with Applications}, 19(3):485--501,
  2012.

\bibitem{trottenberg2000multigrid}
Ulrich Trottenberg, Cornelius~W Oosterlee, and Anton Schuller.
\newblock {\em Multigrid}.
\newblock Elsevier, 2000.

\bibitem{wainwright2008graphical}
Martin~J Wainwright, Michael~I Jordan, et~al.
\newblock Graphical models, exponential families, and variational inference.
\newblock {\em Foundations and Trends{\textregistered} in Machine Learning},
  1(1--2):1--305, 2008.

\bibitem{weiss2000correctness}
Yair Weiss and William~T Freeman.
\newblock Correctness of belief propagation in gaussian graphical models of
  arbitrary topology.
\newblock In {\em Advances in neural information processing systems}, pages
  673--679, 2000.

\bibitem{welling2004choice}
Max Welling.
\newblock On the choice of regions for generalized belief propagation.
\newblock In {\em Proceedings of the 20th conference on Uncertainty in
  artificial intelligence}, pages 585--592. AUAI Press, 2004.

\bibitem{xu2017algebraic}
Jinchao Xu and Ludmil Zikatanov.
\newblock Algebraic multigrid methods.
\newblock {\em Acta Numerica}, 26:591--721, 2017.

\bibitem{yedidia2001bethe}
Jonathan~S. Yedidia, William~T. Freeman, and Yair Weiss.
\newblock Bethe free energy, {K}ikuchi approximations, and belief propagation
  algorithms.
\newblock Technical Report TR2001-16, MERL - Mitsubishi Electric Research
  Laboratories, Cambridge, MA 02139, May 2001.

\bibitem{yedidia2001generalized}
Jonathan~S Yedidia, William~T Freeman, and Yair Weiss.
\newblock Generalized belief propagation.
\newblock In {\em Advances in neural information processing systems}, pages
  689--695, 2001.

\bibitem{yedidia2003understanding}
Jonathan~S Yedidia, William~T Freeman, and Yair Weiss.
\newblock Understanding belief propagation and its generalizations.
\newblock {\em Exploring artificial intelligence in the new millennium},
  8:236--239, 2003.

\bibitem{yedidia2005constructing}
Jonathan~S Yedidia, William~T Freeman, and Yair Weiss.
\newblock Constructing free-energy approximations and generalized belief
  propagation algorithms.
\newblock {\em IEEE Transactions on information theory}, 51(7):2282--2312,
  2005.

\end{thebibliography}

\begin{appendices}
\section{Consistency of GaBP}
\label{Appendices:Consistency_of_GaBP}
Here, following \cite{weiss2000correctness} we prove
\GaBPconsistency*
That is, if there is a steady state under mapping (\ref{nonsymmetric_GaBP_rules}), the solution given by the GaBP rules is exact.

We note that the proof in \cite{weiss2000correctness} also holds for the nonsymmetric case. We present a slightly different version of their reasoning, without referencing  graphical models for normal distribution.

The first concept that we need is a computation tree, which captures the order of operations under the GaBP iteration scheme. The computation tree contains copies of vertices and edges of the graph corresponding to $\mathbf{A}$. The matrix is supposed to be fixed so the computation tree depends on the root node $i$ and the number of steps $n$. We denote it by $T_n(x_i)$. To obtain $T_{n}(x_i)$ from $T_{n-1}(x_i)$, we consider each vertex $m\in\mathcal{V}_{T_{n-1}(x_i)}$ that has no incidence edges, find the corresponding variable on the graph of $A$, add to $T_{n-1}(x_i)$ copies of each neighbour $k$ of $m$ such that $e_{km}\in\mathcal{E}_{A}$ except for $l$ for which $e_{ml}\in\mathcal{E}_{T_{n-1}(x_i)}$. The example of the tree $T_3(x_1)$ is in figure \ref{fig:computation_tree:c}, the $T_2(x_1)$ in the dashed box exemplifies the recursion process.

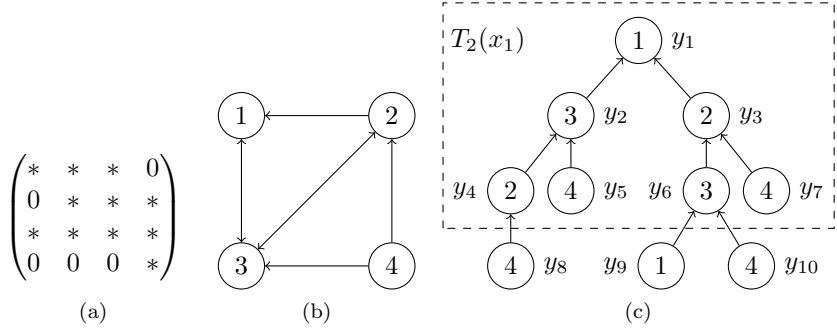
\begin{figure}
    \centering
    \subfloat[]{\label{fig:computation_tree:a}\begin{tikzpicture}[
sq/.style={circle, draw=black, minimum size=1mm},
]
\node[] at (0, 0) (1) {$\begin{pmatrix}*&*&*&0\\0&*&*&* \\ *&*&*&* \\ 0&0&0&* \end{pmatrix}$};
\end{tikzpicture}}\quad
    \subfloat[]{\label{fig:computation_tree:b}\begin{tikzpicture}[
sq/.style={circle, draw=black, minimum size=1mm},
]
\node[sq] at (0, 0) (1) {$1$};
\node[sq] at (2, 0) (2) {$2$};
\node[sq] at (0, -2) (3) {$3$};
\node[sq] at (2, -2) (4) {$4$};
\draw[->] (2) -- (1);
\draw[->] (4) -- (3);
\draw[->] (4) -- (2);
\draw[<->] (1) -- (3);
\draw[<->] (3) -- (2);
\end{tikzpicture}}\quad
    \subfloat[]{\label{fig:computation_tree:c}\begin{tikzpicture}[
sq/.style={circle, draw=black, minimum size=1mm},
]
\node[sq, label=east:$y_1$] at (0, 0) (1) {$1$};
\node[sq, label=east:$y_2$] at (-0.9, -1) (3) {$3$};
\node[sq, label=east:$y_3$] at (+0.9, -1) (2) {$2$};
\node[sq, label=west:$y_4$] at (-1.7, -2) (2_) {$2$};
\node[sq, label=east:$y_5$] at (-0.9, -2) (4) {$4$};
\node[sq, label=west:$y_6$] at (0.9, -2) (3_) {$3$};
\node[sq, label=east:$y_7$] at (1.7, -2) (4_) {$4$};
\node[sq, label=east:$y_8$] at (-1.7, -3) (4__) {$4$};
\node[sq, label=west:$y_9$] at (0.3, -3) (1_) {$1$};
\node[sq, label=east:$y_{10}$] at (1.5, -3) (4___) {$4$};
\node[] at (-2, 0) (empty) {$T_2(x_1)$};
\draw[->] (3) -- (1);
\draw[->] (2) -- (1);
\draw[->] (2_) -- (3);
\draw[->] (4) -- (3);
\draw[->] (3_) -- (2);
\draw[->] (4_) -- (2);
\draw[->] (4__) -- (2_);
\draw[->] (1_) -- (3_);
\draw[->] (4___) -- (3_);
\draw[dashed] (-2.6, -2.5) rectangle ++(5.2, 3);
\end{tikzpicture}}
    \caption{(a) -- matrix with nonzero elements denoted by $*$; (b) -- directed graph corresponding to the matrix. Note that by our convention $e_{ij}$ agrees with $A_{ji}$ not $A_{ij}$; (c) -- computation tree of depth $3$ for the first node $T_3(x_1)$ generated by a flood schedule. The subtree inside the box is $T_2(x_1)$.}
    \label{fig:computation_tree}
\end{figure}

By the construction of the computation tree, the following proposition is true.
\begin{prop}
    If $x_i^{(n)}$ is the solution on the $n$-th step of the algorithm~\ref{algorithm:non_symmetric_GaBP}, then it coincides with the one obtained after the elimination of all variables but $x_i$ (the root) from the computation tree $T_n(x_i)$.
\end{prop}
To relate the matrix $\mathbf{B}$ of the computation tree $T_{n}(x_i)$ to the matrix $\mathbf{A}$, we define the matrix $\mathbf{O}$ \cite[eq. 15]{weiss2000correctness} that connects original variables with copies
\begin{equation}
    \mathbf{y} = \mathbf{O}\mathbf{x}, \mathbf{d} = \mathbf{O}\mathbf{b},
\end{equation}
or, more precisely, $y_j$ is a copy of $x_i \Rightarrow O_{ji}=1$ and $\sum\limits_{i}O_{ji} = 1$. For example, matrix $\mathbf{O}$ for the tree in figure \ref{fig:computation_tree:c} is
\begin{equation}
    \mathbf{O}^T =
    \begin{pmatrix}
        1&0&0&0&0&0&0\\
        0&0&1&1&0&0&0\\
        0&1&0&0&0&1&0\\
        0&0&0&0&1&0&1\\
    \end{pmatrix}.
\end{equation}
Now it is not hard to establish the connection between $\mathbf{B}$ and $\mathbf{A}$ \cite[eq. 17]{weiss2000correctness}
\begin{equation}\label{error_matrix}
    \mathbf{B}\mathbf{O} + \mathbf{E} = \mathbf{O}\mathbf{A},
\end{equation}
where $\mathbf{E}$ is nonzero only for the subset of variables that $n$ steps away from the root node on the computation tree $T_{n}(x_i)$. The final part of the proof depends on the following statement \cite[Periodic beliefs lemma]{weiss2000correctness}.
\begin{prop}
    If there is $N\in\mathbb{N}$ such that $\widetilde{\mu}^{(N+k)}_{e} = \widetilde{\mu}^{(N)}_{e}$, $\widetilde{\Lambda}^{(N+k)}_{e} = \widetilde{\Lambda}^{(N)}_{e}$ for all $e\in\mathcal{E}$ and for any $k\in\mathbb{N}$, then it is possible to construct an arbitrary large computation tree $T_{M}(x_i)$ for any root node $x_i$ such that $\mathbf{O}\boldsymbol{\mu}^{(N)} = \widetilde{\mathbf{B}}^{-1}\widetilde{\mathbf{d}}$. Where $\widetilde{B}_{ij} \neq B_{ij}$ and $\widetilde{d}_i \neq d_i$ only for $i=j$ that are $M$ steps away from the root node.
\end{prop}
The crucial part here is that not only the solution for the root node coincides with the steady state solution of GaBP, but also the same is true for all the variables on the modified computation tree.

The proof is as follows. First, following the recursion procedure, we construct a computation tree of desired depth $M$. Then we continue to grow the tree till the subtrees of nodes $M$ steps away from the root reach the depth $N$ which corresponds to the steady state of GaBP. Now, elimination of subtrees results in the desired modified tree with the matrix $\widetilde{\mathbf{B}}$ and the right-hand side $\widetilde{\mathbf{d}}$.

Since we can construct an arbitrary modified computation tree, we can always get for arbitrary large $M$
\begin{equation}
    \widetilde{\mathbf{B}}\mathbf{O} = \mathbf{O}\mathbf{A}\text{ for the first }M\text{ rows.}
\end{equation}
And we know that by construction of the modified computation tree
\begin{equation}
    \widetilde{\mathbf{B}} \mathbf{O}\boldsymbol{\mu}^{(N)} = \widetilde{\mathbf{d}}.
\end{equation}
So we conclude that
\begin{equation}
    \mathbf{O} \mathbf{A}\boldsymbol{\mu}^{(N)} = \mathbf{O} \mathbf{b} \text{ for the first }M\text{ rows.}
\end{equation}
Note, that $\mathbf{O}^{T}\mathbf{O}$ is a diagonal matrix that counts the number of copies of each variable, therefore we can always choose $M$ large enough to make $\det\left(\mathbf{O}^{T}\mathbf{O}\right)\neq0 $ and $\mathbf{A}\boldsymbol{\mu}^{(N)} = \mathbf{b}$ which means that the iterative scheme defined by the algorithm~\ref{algorithm:non_symmetric_GaBP} is consistent. 

\section{Convergence of GaBP}\label{Appendices:Convergence_of_GaBP}
Here we present the version of the proof from \cite{malioutov2006walk} that extends to nonsymmetric matrices.
Our modifications are relatively minor, but for the sake of logical coherence, we reproduce here the minimal set of arguments from \cite{malioutov2006walk} tuning definitions and proposition when needed.
The main result of this section is

\GaBPconvergence*

The whole idea of the proof \cite{malioutov2006walk} is to relate GaBP operations with the recursive update of the weights of walks on the graph, corresponding to the matrix $\mathbf{A}$. For the start, we define a walk $w$ as a an ordered set of vertices $w = \left(i_1, i_2, \dots, i_{l(w)}\right)$ where $l(w)$ is a length of the walk $w$ and $\forall k<l(w) \Rightarrow e_{i_{k}i_{k+1}}\in\mathcal{E}$. Each walk possesses a weight
\begin{equation}\label{weight_of_the_walk}
    \phi(w) = A_{i_{l(w)} i_{l(w)-1}}\cdots A_{i_3 i_2}A_{i_2 i_1}.
\end{equation}
Note that the order is backward, which is a consequence of our definition of the directed graph. For the symmetric matrix, when the order is not essential, the equation (\ref{weight_of_the_walk}) coincides with the weight defined in \cite[3.1]{malioutov2006walk}. 
Now, if we have a system $\mathbf{A}\mathbf{x} = \mathbf{b}$, we can rescale it using $\widetilde{b}_{j} = b_j/A_{jj}$. This procedure is valid for any $\mathbf{A}$ with nonzero diagonal and results in the equivalent system
\begin{equation}\label{rescaled_system}
    \widetilde{\mathbf{A}} \mathbf{x} = \widetilde{\mathbf{b}},~ \widetilde{A}_{ij} = \delta_{ij} + \left(1 - \delta_{ij}\right)\frac{A_{ij}}{A_{ii}}\equiv \delta_{ij} - \widetilde{R}_{ij}
\end{equation}
It is possible to represent the solution of (\ref{rescaled_system}) in the form of Neumann series (see \cite[ch. 5]{Horn:2012:MA:2422911}) because $\rho\left(\widetilde{\mathbf{R}}\right)<1$ and therefore
\begin{equation}
    \widetilde{\mathbf{A}}^{-1} = \left(\mathbf{I} - \widetilde{\mathbf{R}}\right)^{-1} = \sum_{n=0}^{\infty}\widetilde{\mathbf{R}}^{n}.
\end{equation}
However, for being able to rearrange terms in the sum as necessary, which is sufficient to rewrite the inverse matrix using walks, one needs to require absolute convergence which is $\rho(\widetilde{|\mathbf{R}|})<1$. Having this condition it is not hard to prove \cite[Proposition 1, Proposition 5]{malioutov2006walk}
\begin{prop}
    If $\rho(\widetilde{|\mathbf{R}|})<1$, then $\widetilde{A}^{-1}_{ij} = \sum\limits_{w:j\rightarrow i}\phi(w)$ and $x^{\star}_i \equiv\left(\widetilde{\mathbf{A}}^{-1} \widetilde{\mathbf{b}}\right)_i = \sum\limits_{k\in \mathcal{V}} \sum\limits_{w:k\rightarrow i}\phi(w)\widetilde{b}_k.$
\end{prop}
Here, by $w:j\rightarrow i$ we mean the set of walks which start from the vertex $j$ and end at the vertex $i$. If one defines \cite[3.2]{malioutov2006walk} sets of single-visit $k\overset{\backslash i}{\rightarrow} i$ and single-revisit $i\overset{\backslash i}{\rightarrow} i$ walks by all walks which are not visiting the node $i$ in between given start and end points, the sum over walks can be decomposed \cite[eq. 12, 13; Proposition 9]{malioutov2006walk}
\begin{equation}\label{walk_sum_solution}
    x^{\star}_{i} = \left.\left(\widetilde{b}_i + \sum\limits_{k\in \mathcal{V}} \left[\widetilde{b}_{k}\sum\limits_{w:k\overset{\backslash i}{\rightarrow} i}\phi(w)\right]\right)\middle/\left(1 - \sum\limits_{w:i\overset{\backslash i}{\rightarrow} i}\phi(w)\right)\right..
\end{equation}
The decomposition follows from "topological" considerations alone which depend only on the structure of walks and not on the particular definition of the weight. The last result that we need is \cite[Lemma 18]{malioutov2006walk}
\begin{prop}
    For each finite length walk $k\rightarrow j$ on directed graph of the matrix $A$ there is $n$ and unique walk on the computation tree $T_n(x_i)$.
\end{prop}
Now, if we can relate update rules (\ref{nonsymmetric_GaBP_rules}) with the recursive structure of walks on a tree, the proof of the proposition \ref{prop:GaBP_convergence} is done.

On the tree, for each vertex $i$, the sum over single-revisit walks splits into sums over subtrees $T_{k\cup i}$, which are maximal connected parts that contain $i$ and among $N(i)$, only $k$. Then
\begin{equation}
    \sum\limits_{w:i\overset{\backslash i}{\rightarrow} i}\phi(w) = \sum_{k\in N(i)} \sum\limits_{\substack{w:i\overset{\backslash i}{\rightarrow} i\\w\in T_{k\cup i}}}\phi(w),
\end{equation}
but the sums over subtrees $T_{k\cup i}$ can be written as a sum over $T_{k\backslash i } \equiv T_{k\cup i}\backslash\left\{i\right\}$,
\begin{equation}\label{walks_update_one}
    \sum\limits_{\substack{w:i\overset{\backslash i}{\rightarrow} i\\w\in T_{k\cup i}}}\phi(w) = \frac{\widetilde{R}_{ki} \widetilde{R}_{ik}}{1 - \sum\limits_{\substack{w:k\overset{\backslash k}{\rightarrow} k\\w\in T_{k\backslash i}}}\phi(w)} = \frac{\widetilde{R}_{ki} \widetilde{R}_{ik}}{1 - \sum\limits_{m \in N(k)\backslash i} \sum\limits_{\substack{w:k\overset{\backslash k}{\rightarrow} k\\w\in T_{m\cup k}}}\phi(w)},
\end{equation}
where we used \cite[eq. 12, Proposition 9]{malioutov2006walk}
\begin{equation}
    \sum\limits_{w:k\rightarrow k}\phi(w) = \frac{1}{1 - \sum\limits_{w:k\overset{\backslash k}{\rightarrow} k}\phi(w)}.
\end{equation}
Using the definition of $\widetilde{\mathbf{R}}$, it is easy to see that if one denotes
\begin{equation}\label{precision_parametrization}
    -\frac{A_{ii}}{A_{ki}}\sum\limits_{\substack{w:i\overset{\backslash i}{\rightarrow} i\\w\in T_{k\cup i}}}\phi(w) = \widetilde{\Lambda}_{ki},
\end{equation}
then the update rule (\ref{walks_update_one}) coincides with the one for $\widetilde{\mathbf{\Lambda}}$ in (\ref{nonsymmetric_GaBP_rules}). Note that (\ref{precision_parametrization}) is well defined because if $A_{ki}=0$, there is no contribution from this particular subtree, and we do not need to use the walk from there.
In the same vein, the sum in the numerator of (\ref{walk_sum_solution}) can be decomposed
\begin{equation}
    \sum\limits_{k\in \mathcal{V}} \left[\widetilde{b}_{k}\sum\limits_{w:k\overset{\backslash i}{\rightarrow} i}\phi(w)\right] = \sum_{m\in N(i)} \sum\limits_{k\in T_{m\cup i}} \left[\widetilde{b}_{k}\sum\limits_{\substack{w:k\overset{\backslash i}{\rightarrow} i\\ w\in T_{m\cup i}}}\phi(w)\right].
\end{equation}
Again, using the sum over subtrees $T_{k\backslash i}$
\begin{equation}
    \sum\limits_{k\in T_{m\cup i}} \left[\widetilde{b}_{k}\sum\limits_{\substack{w:k\overset{\backslash i}{\rightarrow} i\\ w\in T_{m\cup i}}}\phi(w)\right] =  \widetilde{R}_{im} \sum\limits_{k\in T_{m\backslash i}} \left[\widetilde{b}_{k}\sum\limits_{\substack{w:k\rightarrow m\\ w\in T_{m\backslash i}}}\phi(w)\right],
\end{equation}
decomposition on single-visit walks \cite[eq. 13]{malioutov2006walk} and equations (\ref{precision_parametrization}), (\ref{walks_update_one}), we obtain
\begin{equation}\label{mean_parametrization}
    \begin{split}
    &\frac{\widetilde{\mu}_{mi}\widetilde{\Lambda}_{mi}}{A_{ii}} \equiv\sum\limits_{k\in T_{m\cup i}} \left[\widetilde{b}_{k}\sum\limits_{\substack{w:k\overset{\backslash i}{\rightarrow} i\\ w\in T_{m\cup i}}}\phi(w)\right] = \\ &= 
    \frac{\widetilde{\Lambda}_{mi}A_{mm}}{A_{ii}}\left(\widetilde{b}_m + \sum\limits_{l\in N(m)\backslash i}\sum\limits_{k\in T_{l \cup m}} \left[\widetilde{b}_{k}\sum\limits_{\substack{w:k\overset{\backslash m}{\rightarrow} m\\w\in T_{l\cup m}}}\phi(w)\right]\right).
    \end{split}
\end{equation}
The parameterization introduced in (\ref{mean_parametrization}) leads to the same update rule for $\widetilde{\boldsymbol{\mu}}$ as in (\ref{nonsymmetric_GaBP_rules}). With that, the sufficient condition, given in proposition \ref{prop:GaBP_convergence}, is established.

\section{Consistency of generalized GaBP}\label{Appendices:Consistency_of_generalized_GaBP}
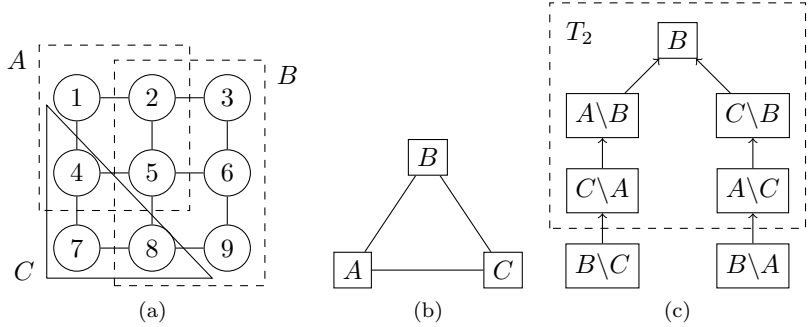
\begin{figure}
    \centering
    \subfloat[]{\label{fig:generalized_computation_tree:a}\begin{tikzpicture}[
sq/.style={circle, draw=black, minimum size=1mm},
]
\node[sq] at (0, 0) (1) {$1$};
\node[sq] at (1, 0) (2) {$2$};
\node[sq] at (2, 0) (3) {$3$};
\node[sq] at (0, -1) (4) {$4$};
\node[sq] at (1, -1) (5) {$5$};
\node[sq] at (2, -1) (6) {$6$};
\node[sq] at (0, -2) (7) {$7$};
\node[sq] at (1, -2) (8) {$8$};
\node[sq] at (2, -2) (9) {$9$};
\draw[-] (1) -- (2);
\draw[-] (1) -- (4);
\draw[-] (2) -- (3);
\draw[-] (2) -- (5);
\draw[-] (3) -- (6);
\draw[-] (4) -- (5);
\draw[-] (4) -- (7);
\draw[-] (5) -- (6);
\draw[-] (5) -- (8);
\draw[-] (6) -- (9);
\draw[-] (7) -- (8);
\draw[-] (8) -- (9);
\draw[dashed] (0.5, -2.5) rectangle ++(2, 3);
\draw[dashed] (-0.5, -1.5) rectangle ++(2, 2.2);
\draw[thin] (-0.4, -2.4) -- (-0.4, -0.1) -- (1.8, -2.4) -- (-0.4, -2.4);
\node[] at (-0.8, 0.5) (empty) {$A$};
\node[] at (2.8, 0.3) (empty) {$B$};
\node[] at (-0.7, -2.3) (empty) {$C$};
\end{tikzpicture}}\quad
    \subfloat[]{\label{fig:generalized_computation_tree:b}\begin{tikzpicture}[
sq/.style={rectangle, draw=black, minimum size=1mm},
]
\node[sq] at (0, 0) (1) {$A$};
\node[sq] at (1, 1.5) (2) {$B$};
\node[sq] at (2, 0) (3) {$C$};
\draw[-] (1) -- (2);
\draw[-] (2) -- (3);
\draw[-] (3) -- (1);
\end{tikzpicture}}\quad
    \subfloat[]{\label{fig:generalized_computation_tree:c}\begin{tikzpicture}[
sq/.style={rectangle, draw=black, minimum size=1mm},
]
\node[sq] at (0, 0) (1) {$A\backslash B$};
\node[sq] at (1, 1) (2) {$B$};
\node[sq] at (2, 0) (3) {$C\backslash B$};
\node[sq] at (2, -1) (3_) {$A\backslash C$};
\node[sq] at (0, -1) (1_) {$C\backslash A$};
\node[sq] at (2, -2) (3__) {$B\backslash A$};
\node[sq] at (0, -2) (1__) {$B\backslash C$};
\draw[->] (1) -- (2);
\draw[->] (3) -- (2);
\draw[->] (3_) -- (3);
\draw[->] (1_) -- (1);
\draw[->] (3__) -- (3_);
\draw[->] (1__) -- (1_);
\draw[dashed] (-0.7, -1.5) rectangle ++(3.4, 3);
\node[] at (-0.3, 1.1) (empty) {$T_2$};
\end{tikzpicture}}
    \caption{(a) -- partition of the original graph on large regions; (b) -- flat representation of the two-layer region graph; (c) -- computation tree for the generalized GaBP.}
    \label{fig:generalized_computation_tree}
\end{figure}
Here we prove that the two-layer generalized GaBP is consistent.
\generalizedGaBPconsistency*
The idea of the proof is the same as for the regular GaBP. One needs to relate, considering the operations of generalized GaBP, equations that the generalized GaBP solves during the $N$-th step, with the original system of linear equations, and then to show that those systems coincide for a sufficiently large $N$ if steady state exists.

To do so, we introduce a flat version of the region graph (an example is shown in figure \ref{fig:generalized_computation_tree:b}) that provides less detailed information about parent-child structure. The flat region graph is an undirected graph $\left\{\mathcal{V}, \mathcal{E}\right\}$, where $\mathcal{V}$ is the set of large regions and $\left(L, L^{'}\right)\in \mathcal{E}$ if $L$ and $L^{'}$ has at least one common child (the example is in figure \ref{fig:generalized_computation_tree:b}).

Now one can introduce the computation tree exactly in the same way as for GaBP. The only difference is that, because of an overlap between large regions, when we add a leaf node, we include overlapping variables to the root node. An example of the computation tree $T_3(B)$ as well as the $T_2(B)$ is in figure \ref{fig:generalized_computation_tree:c}. By construction of the computation tree, we know that the following is true.
\begin{prop}
    Elimination of all the variables on the computation tree $T_N(B)$ leads to the solution $\mathbf{x}_B$ that coincides with the one on the $N$-th step of generalized GaBP.
\end{prop}

The relation between the matrix $\mathbf{B}$, corresponding to the computation tree, and the original matrix $\mathbf{A}$ is the same as in the equation (\ref{error_matrix}) if one introduces the matrix $\mathbf{O}$
\begin{equation}
    O_{ij} = 
    \begin{cases}
        1\text{ if } y_i \text{ is the copy of } x_j,\\
        0\text{ otherwise}.
    \end{cases}
\end{equation}
Here, $\mathbf{x}$ are variables on the graph of  matrix $\mathbf{A}$, and $\mathbf{y}$ are the ones on the computation tree.

Having the same relation between $\mathbf{A}$ and $\mathbf{B}$, we can repeat the rest of the proof, using the same arguments as in Section~\ref{Appendices:Consistency_of_GaBP}. So it follows that generalized GaBP is consistent and proposition \ref{prop:generalized_GaBP_Consistency} is true.

\section{Convergence of generalized GaBP}
\label{Appendices:Convergence_of_generalized_GaBP}
In this section, we present a sufficient condition for the  convergence of the two-layer generalized GaBP. 

\generalizedGaBPconvergence*

The proof consists of two parts. In the first one, we show that single-visit and single-revisit walks on a tree possess the same update rules as generalized GaBP messages. In the second part, we show that it is always possible to reorganize walks on the graph coming from the partition $F$ (see equations (\ref{partition}) and (\ref{matrix_partition})) to restore each walk on a computation tree.
\subsection{Walk structure on a tree}\label{Appendices:Convergence_of_generalized_GaBP:1}
To complete the first part, we define for a given partition $F$ (equation (\ref{partition})) of a matrix $\mathbf{A}$  the weight of a walk $w = \left(i_1 i_2\dots i_{L}\right)$ by the product of matrices
\begin{equation}
    \phi(w) = \widetilde{\mathbf{R}}_{i_{L}i_{L-1}} \cdots \widetilde{\mathbf{R}}_{i_3 i_2} \widetilde{\mathbf{R}}_{i_2 i_1}.
\end{equation}
In the view of the standard result \cite[ch. 8, Theorem 8.9]{amann2005analysis} on absolute convergence in complete finite metric spaces it is possible to rearrange terms of the sum, such that we can formulate the following statement.
\begin{prop}\label{prop:absolute_convergence}
    If $\rho\left(\left\|\widetilde{\mathbf{R}}\right\|\right)<1$, then $\left(\widetilde{\mathbf{A}}^{-1}\right)_{ii} = \sum\limits_{n=0}^{\infty}\left(\widetilde{\mathbf{R}}^{n}\right)_{ii} = \sum\limits_{w:i\rightarrow i}\phi(w)$, $\mathbf{x}_i\equiv \sum\limits_{j\in \mathcal{V}}\left(\widetilde{\mathbf{A}}^{-1}\right)_{ij}\widetilde{\mathbf{b}}_{j} = \sum\limits_{j\in \mathcal{V}} \sum\limits_{w:j\rightarrow i}\phi(w)\widetilde{\mathbf{b}}_j$.
\end{prop}
Here we used the same definition for the set of walks as in the Appendix~\ref{Appendices:Convergence_of_GaBP}. Again, \cite[eq. 12, 13]{malioutov2006walk} allows us to rewrite the diagonal blocks of the inverse matrix and the solution vector using single-visit and single-revisit walks
\begin{equation}
    \begin{split}
        &\left(\widetilde{\mathbf{A}}^{-1}\right)_{ii} = \left(\mathbf{I}_{ii} - \sum\limits_{w:i\overset{\backslash i}{\rightarrow} i}\phi(w)\right)^{-1},\\
        &\mathbf{x}_i = \left(\widetilde{\mathbf{A}}^{-1}\right)_{ii} \left(\widetilde{\mathbf{b}}_i + \sum\limits_{j\in \mathcal{V}}\sum\limits_{w:j\overset{\backslash i}{\rightarrow} i}\phi(w) \widetilde{\mathbf{b}}_j\right).
    \end{split}
\end{equation}
On the tree we can split the sums over  contributions from subtrees $T_{k\cup i}$ for each $k\in N(i)$. Therefore, from comparison with algorithm \ref{algorithm:non_symmetric_two_layers_GaBP}, we can deduce that
\begin{equation}\label{generalized_message_parametrization}
    \sum_{j\in T_{k\cup i}} \sum\limits_{\substack{w:j\overset{\backslash i}{\rightarrow} i\\ w\in T_{k\cup i}}}\phi(w)\widetilde{\mathbf{b}}_j = \left(\mathbf{A}_{ii}\right)^{-1}\mathbf{m}_{ki},~\sum\limits_{\substack{w:i\overset{\backslash i}{\rightarrow} i\\ w\in T_{k\cup i}}}\phi(w) = - \left(\mathbf{A}_{ii}\right)^{-1}\mathbf{\Lambda}_{ki}.
\end{equation}
Messages in algorithm \ref{algorithm:non_symmetric_two_layers_GaBP} propagate along edges of the region graph, whereas messages that we have just defined flow along edges of a graph of the matrix $\widetilde{\mathbf{R}}$. To have a more straightforward connection between them, we consider $\widetilde{\mathbf{R}}$ as a matrix originates from the computation tree itself. Under this set of circumstances, there is a one-to-one correspondence between messages (\ref{generalized_message_parametrization}) and the ones in algorithm~\ref{algorithm:non_symmetric_two_layers_GaBP}.

For single-revisit walks, one has
\begin{equation}\label{generalized_update_lambda}
    \begin{split}
        \sum\limits_{\substack{w:i\overset{\backslash i}{\rightarrow} i\\ w\in T_{k\cup i}}}\phi(w) &= \left(\mathbf{A}_{ii}\right)^{-1} \mathbf{A}_{ik} \sum\limits_{\substack{w:k\rightarrow k\\ w\in T_{k\backslash i}}}\phi(w) \left(\mathbf{A}_{kk}\right)^{-1}\mathbf{A}_{ki} = \\
        &= \left(\mathbf{A}_{ii}\right)^{-1} \mathbf{A}_{ik}\left(\mathbf{I}_{kk} - \sum\limits_{\substack{w:k\overset{\backslash k}{\rightarrow} k\\ w\in T_{k\backslash i}}}\phi(w) \right)^{-1}\left(\mathbf{A}_{kk}\right)^{-1}\mathbf{A}_{ki} \Rightarrow \\
        &\Rightarrow \mathbf{\Lambda}_{ki} = -\mathbf{A}_{ik} \left(\mathbf{A}_{kk} + \sum\limits_{m\in N(k)\backslash i}\mathbf{\Lambda}_{mk}\right)^{-1}\mathbf{A}_{ki}.
    \end{split}
\end{equation}
If we consider the following matrix
\begin{equation}
    \left(\begin{pmatrix}
        0 & \mathbf{A}_{ik}\\
        \mathbf{A}_{ki} & \left(\mathbf{A}_{kk} + \sum\limits_{m\in N(k)\backslash i} \mathbf{\Lambda}_{mk}\right)\\
    \end{pmatrix}^{-1}\right)_{ii} = \mathbf{\Lambda}_{ki}^{-1},
\end{equation}
one can immediately see that update rules (\ref{generalized_update_lambda}) indeed coincide with (\ref{rough_message_update}).

For single-visit walks, we have
\begin{equation}
    \begin{split}
        &\sum_{j\in T_{k\cup i}} \sum\limits_{\substack{w:j\overset{\backslash i}{\rightarrow} i\\ w\in T_{k\cup i}}}\phi(w)\widetilde{\mathbf{b}}_j = -\left(\mathbf{A}_{ii}\right)^{-1}\mathbf{A}_{ik} \sum_{j\in T_{k\backslash i}} \sum\limits_{\substack{w:j\rightarrow k\\ w\in T_{k\backslash i}}}\phi(w)\widetilde{\mathbf{b}}_j =\\
        &=-\left(\mathbf{A}_{ii}\right)^{-1} \mathbf{A}_{ik}\left(\mathbf{I}_{kk} - \sum\limits_{\substack{w:k\overset{\backslash k}{\rightarrow} k\\ w\in T_{k\backslash i}}}\phi(w) \right)^{-1}\left(\widetilde{\mathbf{b}}_{k} + \sum_{j\in T_{k\backslash i}} \sum\limits_{\substack{w:j\overset{\backslash k}{\rightarrow} k\\ w\in T_{k\backslash i}}}\phi(w)\widetilde{\mathbf{b}}_j\right),
    \end{split}
\end{equation}
or using (\ref{generalized_message_parametrization}), we get
\begin{equation}\label{generalized_update_mu}
    \mathbf{m}_{ki} = -\mathbf{A}_{ik}\left(\mathbf{A}_{kk} + \sum\limits_{m\in N(k)\backslash i}\mathbf{\Lambda}_{mk}\right)^{-1}\left(\mathbf{b}_{k} + \sum\limits_{p\in N(k)\backslash i} \mathbf{m}_{pk}\right).
\end{equation}
Since $\mathbf{m}_{ki} = \mathbf{\Lambda}_{ki}\boldsymbol{\mu}_{ki}$ and
\begin{equation}
   \boldsymbol{\mu}_{ki} = \left(\begin{pmatrix}
        0 & \mathbf{A}_{ik}\\
        \mathbf{A}_{ki} & \left(\mathbf{A}_{kk} + \sum\limits_{m\in N(k)\backslash i} \mathbf{\Lambda}_{mk}\right)\\
    \end{pmatrix}^{-1}
    \begin{pmatrix}
    0\\
    \mathbf{b}_k + \sum\limits_{p\in N(k)\backslash i} \mathbf{m}_{pk}
    \end{pmatrix}\right)_{ii}
\end{equation}
we recover update rules (\ref{rough_message_update}). So we conclude that on the computation tree update rules for the two-layer generalized GaBP coincide with recursive relations for the single-visit and single-revisit walks.
\subsection{Walk-sums and the graph refinement}\label{Appendices:Convergence_of_generalized_GaBP:2}
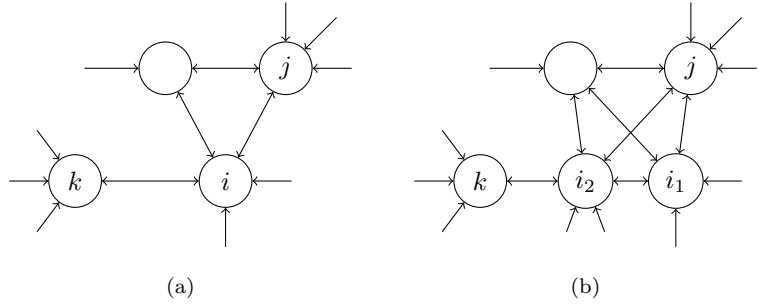
\begin{figure}
    \centering
    \subfloat[]{\label{fig:graph_refinement:a}\begin{tikzpicture}[
sq/.style={circle, draw=black, minimum size=7mm},
]
\node[sq] at (0, 0) (1) {$i$};
\node[sq] at (-0.8, 1.5) (2) {};
\node[sq] at (0.8, 1.5) (3) {$j$};
\node[sq] at (-2, 0) (4) {$k$};
\node[] at (-3, 0) (empty_4) {};
\node[] at (-2.6, 0.8) (empty_4_) {};
\node[] at (-2.6, -0.8) (empty_4__) {};
\node[] at (1, 0) (empty_1) {};
\node[] at (0, -1) (empty_1_) {};
\node[] at (-2, 1.5) (empty_2) {};
\node[] at (1.8, 1.5) (empty_3) {};
\node[] at (0.8, 2.5) (empty_3_) {};
\node[] at (1.6, 2.3) (empty_3__) {};
\draw[<->] (1) -- (2);
\draw[<->] (1) -- (3);
\draw[<->] (2) -- (3);
\draw[<->] (1) -- (4);
\draw[<-] (4) -- (empty_4__);
\draw[<-] (4) -- (empty_4_);
\draw[<-] (4) -- (empty_4);
\draw[<-] (1) -- (empty_1_);
\draw[<-] (1) -- (empty_1);
\draw[<-] (2) -- (empty_2);
\draw[<-] (3) -- (empty_3);
\draw[<-] (3) -- (empty_3_);
\draw[<-] (3) -- (empty_3__);
\end{tikzpicture}}\quad
    \subfloat[]{\label{fig:graph_refinement:b}\begin{tikzpicture}[
sq/.style={circle, draw=black, minimum size=7mm},
]
\node[sq] at (0.6, 0) (1) {$i_1$};
\node[sq] at (-0.6, 0) (_1) {$i_2$};
\node[sq] at (-0.8, 1.5) (2) {};
\node[sq] at (0.8, 1.5) (3) {$j$};
\node[sq] at (-2, 0) (4) {$k$};
\node[] at (-3, 0) (empty_4) {};
\node[] at (-2.6, 0.8) (empty_4_) {};
\node[] at (-2.6, -0.8) (empty_4__) {};
\node[] at (1.6, 0) (empty_1) {};
\node[] at (0.6, -1) (empty_1_) {};
\node[] at (-0.3, -0.8) (empty__1) {};
\node[] at (-0.9, -0.8) (empty__1_) {};
\node[] at (-2, 1.5) (empty_2) {};
\node[] at (1.8, 1.5) (empty_3) {};
\node[] at (0.8, 2.5) (empty_3_) {};
\node[] at (1.6, 2.3) (empty_3__) {};
\draw[<->] (1) -- (2);
\draw[<->] (1) -- (3);
\draw[<->] (_1) -- (2);
\draw[<->] (_1) -- (3);
\draw[<->] (2) -- (3);
\draw[<->] (_1) -- (4);
\draw[<->] (1) -- (_1);
\draw[<-] (4) -- (empty_4__);
\draw[<-] (4) -- (empty_4_);
\draw[<-] (4) -- (empty_4);
\draw[<-] (1) -- (empty_1_);
\draw[<-] (1) -- (empty_1);
\draw[<-] (2) -- (empty_2);
\draw[<-] (3) -- (empty_3);
\draw[<-] (3) -- (empty_3_);
\draw[<-] (3) -- (empty_3__);
\draw[<-] (_1) -- (empty__1);
\draw[<-] (_1) -- (empty__1_);
\end{tikzpicture}}
    \caption{(a) -- graph of the matrix $\mathbf{A}$, each node corresponds to the diagonal block; (b) -- refined version of (a), submatrix $\mathbf{A}_{ii}$ is split by four blocks $\mathbf{A}_{i_1i_1}, \mathbf{A}_{i_1i_2}, \mathbf{A}_{i_2i_1},\mathbf{A}_{i_2i_2}$.}
    \label{fig:graph_refinement}
\end{figure}
The second part of the proof establishes the connection between sets of walks on the graph of the matrix $\widetilde{\mathbf{R}}$ and walks on the computation tree. First, for the matrix (\ref{matrix_partition}) we split a single region $i$ into two parts $i_1$ and $i_2$
\begin{equation}\label{refined_matrix_partition}
    \mathbf{A} = 
    \begin{pmatrix}
        \mathbf{A}_{i_{1}i_{1}}&\mathbf{A}_{i_{1}i_{2}}&\mathbf{A}_{i_{1}j}&\hdots\\
        \mathbf{A}_{i_{2}i_{1}}&\mathbf{A}_{i_{2}i_{2}}&\mathbf{A}_{i_{2}j}&\hdots\\
        \mathbf{A}_{ji_{1}}&\mathbf{A}_{ji_{2}}&\mathbf{A}_{jj}&\hdots\\
        \vdots&\vdots&\vdots&\ddots\\
    \end{pmatrix},
    \mathbf{b} = 
    \begin{pmatrix}
        \mathbf{b}_{i_1}\\
        \mathbf{b}_{i_2}\\
        \mathbf{b}_{j}\\
        \vdots\\
    \end{pmatrix}.
\end{equation}
The transformation of the graph is in figure \ref{fig:graph_refinement}. We refer to this procedure as to the elementary refinement of the region $i$. From the construction of the refined matrix $\mathbf{A}$, the following proposition holds.
\begin{prop}
    There is a one-to-one correspondence between walks on the graph of $\widetilde{\mathbf{R}}$ and the one obtained by the elementary refinement of the region $i$ excluding three situations: 1) walk crosses $i$, 2) walk ends at $i$, 3) walk starts at~$i$.
\end{prop}
We discuss each of these situations separately. First, we need to introduce a new notation. Let $k\overset{M}{\longrightarrow} l$ be the set of walks, where each walk starts from $k$, ends at $l$ and newer leaves the subset $M$. It is easy to see that on the refined graph
\begin{equation}\label{subwalks_on_refined_graph}
    \phi\left(k\overset{\left\{i_1, i_2\right\}}{\longrightarrow} l\right) = \left(\left(\mathbf{A}_{ii}\right)^{-1}\right)_{l k}\mathbf{A}_{kk},\text{ where }l,k=\left\{i_1,i_2\right\}.
\end{equation}

\begin{itemize}
    \item Walk on $\widetilde{\mathbf{R}}$ that crosses $i$  has a form $w_{\text{cross}} = \left(\dots j i k \dots\right)$ (see figure \ref{fig:graph_refinement:a}). The weight of this walk is
\begin{equation}
    \phi(w_{\text{cross}}) = \cdots \left(\mathbf{A}_{kk}\right)^{-1}\mathbf{A}_{ki}\left(\mathbf{A}_{ii}\right)^{-1} \mathbf{A}_{ij} \cdots.
\end{equation}
On the refined graph we can consider the set of all walks that coincides with $w$ outside $i$. The sum  of weight of all these walks is
\begin{equation}
    \phi(w)_{\text{refined}} = \sum\limits_{l, k \in \left\{i_1, i_2\right\}}\cdots \left(\mathbf{A}_{kk}\right)^{-1} \mathbf{A}_{kl}\phi\left(l\overset{\left\{i_1, i_2\right\}}{\longrightarrow} k\right) \left(\mathbf{A}_{ll}\right)^{-1}\mathbf{A}_{lj}\cdots.
\end{equation}
We see that due to equation (\ref{subwalks_on_refined_graph}), weights are the same.

\item Walk on $\widetilde{\mathbf{R}}$ that ends at $i$ has a form $w_{\text{end}}=\left(\dots j i\right)$ and a weight
\begin{equation}
    \phi(w) = \left(\mathbf{A}_{ii}\right)^{-1}\mathbf{A}_{ij}\cdots.
\end{equation}
On the refined graph we have two set of walks 
\begin{equation}
    \phi(w)^{p}_{\text{refined}} = \sum\limits_{l\in\left\{i_1, i_2\right\}}\phi\left(l\overset{\left\{i_1, i_2\right\}}{\longrightarrow} p\right)\left(\mathbf{A}_{ll}\right)^{-1}\mathbf{A}_{lj}\cdots,~p\in\left\{i_1, i_2\right\}
\end{equation}
that can be combined to have the same weight. Namely, using (\ref{subwalks_on_refined_graph}) we find that
\begin{equation}
    \left[\left(\mathbf{A}_{ii}\right)^{-1}\mathbf{A}_{ij}\cdots\right]_{l\star} = \left(\phi(w)^{l}_{\text{refined}}\right)_{\star},~l=\left\{i_1, i_2\right\}.
\end{equation}

\item Walk on $\widetilde{\mathbf{R}}$ that starts at $i$ has a form $w_{\text{start}}=\left(i j \dots\right)$ and a weight
\begin{equation}
    \phi(w_{\text{start}}) = \cdots\left(\mathbf{A}_{jj}\right)^{-1}\mathbf{A}_{ji}.
\end{equation}
It is possible to relate this walk to two sets of walks $w_1 = (i_1j\dots)$, $w_2 = (i_2j\dots)$ on the refined graph multiplying by the corresponding inverse matrices
\begin{equation}\label{reweight}
    \left(\phi(w_{\text{start}})\left(\mathbf{A}_{ii}\right)^{-1}\right)_{\star l} = \sum\limits_{k=\left\{i_1, i_2\right\}}\left(\phi(w_{k})\phi\left(l\overset{\left\{i_1, i_2\right\}}{\longrightarrow} k\right)\left(\mathbf{A}_{ll}\right)^{-1}\right)_{\star},
\end{equation}
where $l=\left\{i_1, i_2\right\}$. The re-weight is needed because the original linear system and the refined one are multiplied by different block diagonal matrices and have different inverses.
\end{itemize}
We know the following two propositions to be true.
\begin{prop}
    Any computation tree can be, by the set of elementary refinements, turned to a computation tree of GaBP under a proper schedule (see discision before \cite[Lemma 18]{malioutov2006walk}) operating on the graph of the matrix (\ref{matrix_partition}) partitioned according to $F$.
\end{prop}
\begin{prop}
    For each walk on the graph of the matrix (\ref{matrix_partition}), there is a unique walk on a sufficiently large computation tree formed by a proper schedule.
\end{prop}
Hence for each walk on the computation tree, it is always possible to find a unique set of walks on the graph of the matrix (\ref{matrix_partition}) that has the same weight after the multiplication by an appropriate inverse matrix (see \ref{reweight}). It allows us to conclude that if it is possible to define a walk-sum for matrix (\ref{matrix_partition}) (see proposition \ref{prop:absolute_convergence}), walk-sum on the computation tree converges too, so the proposition \ref{prop:Generalized_GaBP_convergence} is proven.
\end{appendices}
\end{document}